\documentclass[11pt]{article}
\usepackage{amsmath, graphicx, amsfonts,amssymb}
\usepackage{amsfonts,mathrsfs, color, amsthm}
\usepackage{bm}

\def\Rn{\mathbb{R}^n}

\usepackage[bookmarks=true,
bookmarksnumbered=true, breaklinks=true,
pdfstartview=FitH, hyperfigures=false,
plainpages=false, naturalnames=true,
colorlinks=true,pagebackref=false,
pdfpagelabels]{hyperref}

\usepackage{enumitem}
\usepackage{hyperref}
\hypersetup{
  colorlinks   = true,  
  urlcolor     = blue,  
  linkcolor    = blue, 
  citecolor   = red    }
  
\newtheorem{theorem}{Theorem}[section]
\newtheorem{definition}[theorem]{Definition}
\newtheorem{lemma}[theorem]{Lemma}
\newtheorem{corollary}[theorem]{Corollary}

\newtheorem{proposition}[theorem]{Proposition}

\newtheorem{problem}[theorem]{Problem}

\numberwithin{equation}{section}
 \def\LC{{\mathrm{LC}}_n}

 \def\R{\mathbb{R}}
 
 \def\C{\mathfrak C}
 \def\dom{{\mathrm {dom}}}

 \def\sphere{S^{n-1}}
 \def\R{\mathbb{R}}
 
\def\Rma{\mathfrak{R}_{\alpha}}
\def\Ia{\mathcal{I}_{\alpha}}

\begin{document}

\title{{The Riesz $\alpha$-energy of log-concave functions and related Minkowski problem} }

\author{Niufa Fang,   Deping Ye and Zengle Zhang}
\date{}

\maketitle

\begin{abstract}We calculate the first order variation of the Riesz $\alpha$-energy of a log-concave function $f$ with respect to the Asplund sum. Such a variational formula induces the Riesz $\alpha$-energy measure of log-concave function $f$, which will be denoted by $\mathfrak{R}_{\alpha}(f, \cdot)$. We pose the related Riesz $\alpha$-energy Minkowski problem aiming  to find necessary and/or sufficient conditions on a pregiven Borel measure $\mu$ defined on $\Rn$ so that $\mu=\mathfrak{R}_{\alpha}(f,\cdot)$ for some log-concave function $f$. Assuming enough smoothness, the Riesz  $\alpha$-energy Minkowski problem reduces to a new Monge-Amp\`{e}re type equation involving the Riesz $\alpha$-potential. Moreover, this new Minkowski problem can be viewed as a functional  counterpart  of the recent   Minkowski problem for the chord measures in integral geometry posed by Lutwak, Xi, Yang and Zhang (Comm.\ Pure\ Appl.\ Math.,\ 2024). The Riesz $\alpha$-energy Minkowski problem will be solved under certain mild conditions on $\mu$. 

\end{abstract}

\section{Introduction} 
The variational formula regrading log-concave functions has received a lot of attention recently. Hereafter, by a log-concave function, we mean a function $f:\mathbb{R}^n\to[0,\infty)$ such that $\varphi=-\log f:\mathbb{R}^n\to\mathbb{R}\cup \{+\infty\}$ is convex, namely $$\varphi(\lambda x+(1-\lambda)y) \leq \lambda\varphi( x)+(1-\lambda)\varphi(y)$$ holds for any $x, y\in \Rn$ and $\lambda\in [0, 1]$.   Let $\LC$ be the set of upper semi-continuous log-concave functions such that the total mass of $f$ is nonzero and finite, i.e., $$0<J(f)=\int_{\Rn} f(x)dx<\infty.$$
The story of calculating  the variational formula regrading log-concave functions starts from the seminal work by Colesanti and Fragal\`a in \cite{CF13}, where they established some variational formulas of $J(f)$ in terms of the Asplund sum  $f\oplus t\bullet g$. Hereafter, for two log-concave functions $f\!=\!e^{-\varphi}, g\!=\!e^{-\psi}$ and $t>0$,   
\begin{equation}\label{def-asp-sum-1}
	(f\oplus t\bullet g) = e^{-(\varphi^*+t\psi^*)^*},
\end{equation} where $\phi^*$ denotes the Legendre transform of $\phi:\R^n\to\R \cup \{+\infty\}$ (not necessarily convex), i.e., 
\begin{equation}\label{def-dual}
\phi^*(y)=\sup_{x\in\mathbb{R}^n}\left\{\langle x,y\rangle-\phi(x)\right\} \ \ \ \mathrm{for} \ \ y\in\mathbb{R}^n,
\end{equation} where $\langle x, y\rangle $ is the inner product of $x,y\in \R^n$.  The results of  Colesanti and Fragal\`a \cite{CF13} require certain regularity on functions $f$ and $g$, and in particular the requirement that $g$ is an admissible perturbation of $f$, namely, there exists a constant $c_1>0$ such that $\varphi^\ast-c_1\psi^\ast$ is a convex function. Fang, Xing and Ye in \cite{FXY20+} obtained the variational formula for the total mass in term of the $L_p$ Asplund sum of log-concave functions.

Rotem in \cite{Rot20, Rotem2022} dropped the regularity assumptions and the admissible-perturbation, and obtained the following more general variational formula.  For any $f\in\LC$, let $K_f=\overline{\{x\in\mathbb{R}^n: f(x)>0\}}$, where $\overline{E}$ denotes the closure of $E\subset\Rn.$  Note that  $K_f$ is a closed convex set and its boundary $\partial K_f$ is a  Lipschitz manifold. Therefore, the Gauss map $\nu_{K_f}$ is defined $\mathcal{H}^{n-1}$-almost everywhere (where $\,\mathcal{H}^{n-1}$ is the $(n-1)$-dimensional Hausdorff measure) and the support function of $K_f$  can be  defined as $h_{K_f}(u)=\sup_{x\in K_f}\langle u,x\rangle$ for $u\in S^{n-1}$.

\begin{theorem}\label{Rotem-11}
Let $f=e^{-\varphi}\in \LC$ and $g=e^{-\psi}\in\LC$.  Then
\begin{align}
\lim_{t\to 0^+}\frac{J(f\oplus t\bullet g)-J(f)}{t}
=\int_{K_f}\psi^\ast(\nabla \varphi(x)) f(x)\,dx+\int_{\partial K_f}h_{K_g}(\nu_{K_f}(x)) f(x)\,d\mathcal{H}^{n-1}(x), \label{formula-Rotem}
\end{align} where    $\nabla \varphi$ is the gradient of $\varphi$.  
\end{theorem}  
Theorem \ref{Rotem-11} gives ideas about how a typical variational formula for log-concave functions looks like. Note that the first integral in the right side of \eqref{formula-Rotem}  gives a Borel measure $\mu_f$ defined on $\Rn$, the so-called surface area measure of $f$ \cite{CF13} (or the moment measure used in \cite{CK2015}): \begin{align}\label{surface-a-f} \int_{K_f} \mathbf{h}(\nabla \varphi(x)) f(x)\,dx= \int_{\mathbb{R}^n}\mathbf{h}(y)  \,d\mu_f(y),\end{align}
for each Borel function $\mathbf{h}:\mathbb{R}^n\to\mathbb{R}$ such that $\mathbf{h}\in\mathcal{L}^1(\mu_f(\cdot))$ or $\mathbf{h}$ is non-negative. The second integral in \eqref{formula-Rotem} generates a Borel measure defined on the unit sphere $\sphere$, i.e., \begin{align}\label{surface-a-f-s} 
 \int_{\sphere}\mathbf{g}(u)\, d\mu_f^s(f)=\int_{\partial K_f}\mathbf{g}(\nu_{K_f}(x)) f(x)\,d\mathcal{H}^{n-1}(x), 
\end{align}  for each Borel function $\mathbf{g}:\sphere \to\R$ such that $\mathbf{g}\in\mathcal{L}^1(\mu_f^s(\cdot))$ or $\mathbf{g}$ is non-negative.

 To obtain Theorem \ref{Rotem-11}, Rotem used the representation formula of variation measures for bounded variation functions and also used the geometric variational formulas of convex bodies (employing to super-level sets of $f\in \LC$). Rotem's techniques have been successfully applied to the  $(q-n)$-th moment of $f\in \LC$ in \cite{HLXZ} by Huang, Liu, Xi and Zhao, and to the Orlicz moment of $f\in \LC$ by Fang, Ye, Zhang and Zhao  \cite{FYZZ}. However, the variational results in \cite{FYZZ,HLXZ} require the following  additional condition: for some $\alpha_0\in(0,1)$,
\begin{align}
\limsup_{|x|\to0}\frac{|f(x)-f(o)|}{|x|^{\alpha_0+1}}<+\infty \label{addition-at-o},
\end{align} where $|x|$ denotes the Euclidean norm of $x\in \Rn$. We would like to mention that additional condition \eqref{addition-at-o} is mainly used to control the position of the origin $o\in \Rn$ so that  the origin $o$ sits in the interiors of all super-level sets of $f\in \LC$ in \cite{FYZZ,HLXZ}. Very recently,  Ulivelli in \cite{Ulivelli}, found another way to prove the variational formula of the $(q-n)$-th moments of log-concave function $f\in \LC$ in terms of the Asplund sum \eqref{def-asp-sum-1} where Ulivelli was able to remove the extra condition \eqref{addition-at-o}. Ulivelli's approach is again based on geometric variational formulas for convex bodies $K^{\varphi}$ constructing from the given convex function $\varphi=-\log f$ whose domain is assumed to be compact. By an approximation approach, Ulivelli proves the following variational formula (see \cite[Theorem 1.3]{HLXZ} for the statement with the extra condition \eqref{addition-at-o}). Note that, again like in \eqref{surface-a-f} and \eqref{surface-a-f-s}, the integral expressions in Theorem \ref{HLXZ-Theo1} produce a Borel measure defined on $\Rn$  and a Borel measure defined on $\sphere,$ respectively, (see \cite{HLXZ} for more details).

\begin{theorem}\label{HLXZ-Theo1} 
   Let $q>0$, $f=e^{-\varphi}\in \LC$ and $g=e^{-\psi}\in\LC$. Suppose that $\widetilde{V}_q(f):=\int_{\R^n}f(x)|x|^{q-n}dx>0$ and $K_g$ is compact. If $\varphi(o),\psi(o)<+\infty$, then
\begin{align*}
\lim_{t\to 0^+}\frac{\widetilde{V}_q(f\oplus t\bullet g)-\widetilde{V}_q(f)}{t}
=\!\!\int_{K_f}\!\psi^\ast(\nabla\varphi(x)) f(x)|x|^{q-n}\,dx \!+\!\int_{\partial K_f}\!\!h_{K_g}(\nu_{K_f}(x))   f(x)|x|^{q-n}\,d\mathcal{H}^{n-1}(x).
\end{align*}  
\end{theorem}

Notice that, in general, the conditions on $f, g$ in the variational formulas \cite{FYZZ,HLXZ, Ulivelli} may not cover the case when $g=f$
; while the variational formulas in \cite{CF13, 
FXY20+} require high regularities although do cover the case $g=  f$.  A major goal in this paper is to find arguably more suitable conditions on $f, g\in \LC$ so that the related variational formulas can cover the case $g=  f$ but also require less  or even no regularities. In order to fulfill the goal, let us take a look at how the geometric variational formulas for convex bodies work in e.g., \cite{GHWXY19, HLYZ16, LYZ18, XY17}. When the radial functions of convex bodies are involved, one always uses the perturbation on a convex  body $K$ containing the origin $o$ in its interior: $f_t(u)=h_K(u)+t\mathfrak{g}(u)$ for $u\in S^{n-1}$, where $h_K(u)=\max_{x\in K}\langle x, u\rangle$ is the support function of $K$, $\mathfrak{g}$ is a continuous function on $\sphere$ and $t$ is close to $0$. Note that the support function $h_K$ is  positive  on $\sphere$, and hence a constant $a>0$ can be found such that  \begin{align}
   -\infty <\inf_{u\in \sphere} \mathfrak{g}(u)\leq \mathfrak{g}(u) \leq a h_K(u) \quad  \text{for all}\   u\in \sphere. \label{compare-K-g}
\end{align}  When $\mathfrak{g}=h_L$ for some compact convex set $L$,  
the right-most inequality in \eqref{compare-K-g} gives 
\begin{align} \label{KsubL}
L\subseteq a K. \end{align}

Back to the case of log-concave functions, it is natural to have a condition similar to \eqref{compare-K-g} and an assumption similar to $o\in \mathrm{int}(K)$, the interior of $K$. For log-concave function $f=e^{-\varphi}$, its support function $h_f$ is indeed equal to $\varphi^*$. So a natural replacement of \eqref{compare-K-g} is that the log-concave functions  $f=e^{-\varphi}$ and $g=e^{-\psi}$ satisfy the following condition: there exist two constants $\beta_1>0$ and $\beta_2\in\R$ such that \begin{align}
   -\infty <\inf \psi^*\leq \psi^\ast\le \beta_1\varphi^\ast+\beta_2\ \ \ \mathrm{on} \ \ \R^n. \label{nat-condition-1} 
\end{align} (The need of $\beta_2$ is due to the nature of log-concave functions: $f=e^{-(\varphi+a)} =e^{-a}e^{-\varphi}$.) We take use of $o\in \mathrm{int}(K_f)$ to control the position of $o$, implying that $e^{-\varphi^*}$ is also belonging to $\LC$. As one can see in later context, this is a natural assumption. In particular, the upper bound in \eqref{nat-condition-1} yields  $\psi\geq   \varphi \beta_1-\beta_2$ (where $(\varphi\beta_1)(x)=\beta_1\varphi(x/\beta_1)$).   
This implies that $\dom (\psi)\subseteq \beta_1\dom (\varphi)$, or equivalently $K_g\subseteq \beta_1K_f$, resembling  condition \eqref{KsubL} for convex bodies. 
Note that  condition \eqref{nat-condition-1} does cover the case when $g=  f$, which follows easily from the facts that $h_f=h_g$ and $\inf h_g=  \inf h_f\geq -\varphi(o)>-\infty.$ 

Let us pause here to comment that, when $o\in \mathrm{int}(K_f)$,  the conditions on log-concave functions in all previously mentioned variational formulas do satisfy  condition \eqref{nat-condition-1}.  For example, if $f, g\in \LC$ satisfy the regularity conditions in  \cite{CF13} (mainly $f, g$ are continuous differentiable) and $g=e^{-\psi}$ is an admissible perturbation of $f\!=\!e^{-\varphi}$, then  $\varphi^\ast\!-\!c_1\psi^\ast$ is  convex  (with $c_1\!>0$ a given constant) and  
\begin{align}\label{A1}(\varphi^*-c_1\psi^*)(y)\geq (\varphi^*-c_1\psi^*)(o)+\langle y, \nabla \varphi^*(o)-c_1\nabla\psi^*(o)\rangle\geq c_2-c_3|y|
\end{align}
where $c_2=(\varphi^*-c_1\psi^*)(o)$ and $c_3=|\nabla \varphi^*(o)-c_1\nabla\psi^*(o)|.$  
As $o\in \mathrm{int}(K_f)$, one gets $e^{-\varphi^*}$ is integrable and hence $\varphi^*(y)\geq c_4|y|+c_5$ for some constants $c_4>0$ and $c_5\in \R$ (see \eqref{In.CF13-11}). This, together with \eqref{A1}, gives $$\psi^*\leq \Big(\frac{c_3+c_4}{c_1c_4}\Big)\varphi^*-\frac{c_3c_5+c_2c_4}{c_1c_4}.$$ That is, the admissible perturbation along with the regularity assumption in \cite{CF13} yields \eqref{nat-condition-1}. On the other hand, when $g\!=\! e^{-\psi}$ is supported on a compact set $E\subset\Rn$, then  $\psi(x)=\infty$ if  $x\notin E$ and   \begin{align*}
\psi^\ast(y) =\sup_{x\in \Rn} \Big\{\langle x, y \rangle -\psi(x) \Big\} =\sup_{x\in E} \Big\{\langle x, y \rangle -\psi(x) \Big\} \leq d_0|y|-\min_{x\in E}\psi(x),\end{align*} where $d_0$ a constant such that $E\subset\{x\in \Rn: |x|\leq d_0\}$. Again using \eqref{In.CF13-11}, one gets that $f$ and $g$ satisfy   condition \eqref{nat-condition-1}. 

 In the main context of this paper, we will use the fundamental Reisz $\alpha$-energy as an example to illustrate how \eqref{nat-condition-1} along with $o\in \mathrm{int}(K_f)$ can be used to establish the variational formula  for log-concave functions.  Our approach can also be used  in \cite{FYZZ,HLXZ} to get rid  of condition \eqref{addition-at-o}.  Although the Reisz $\alpha$-energy can be defined for more general functions, we shall concentrate on the log-concave functions.  Let $f: \Rn\rightarrow [0, \infty)$ be a log-concave function  and $\alpha>0$ be a fixed real number. The Riesz $\alpha$-potential of $f$, denoted by $I_{\alpha}(f,\cdot)$, is given by: for $y\in \Rn$,  \begin{align}
I_{\alpha}(f, y)= \int_{\mathbb{R}^n}\frac{f(x)}{|x-y|^{n-\alpha}}\,dx. \label{Riesz-P-def-1}\end{align} The Riesz $\alpha$-energy of $f$ is formulated by 
\begin{align*} 
\Ia(f)=\int_{\mathbb{R}^n}I_{\alpha}(f,y) f(y)\,dy=\int_{\mathbb{R}^n}\int_{\mathbb{R}^n}\frac{f(x)f(y)}{|x-y|^{n-\alpha}}\,dx\,dy.\end{align*}  The importance of both the Riesz $\alpha$-potential  and  the Riesz $\alpha$-energy can never be overemphasized, because they are fundamental notions in mathematics and related areas. Here we give two examples to explain how the Riesz $\alpha$-potential and Riesz $\alpha$-energy are useful in applications. The first example is $\mathcal{I}_2(f)$ with $f$ being a locally integrable function in $\mathbb{R}^n$. It is known that $\mathcal{I}_2(f)$ is the Coulomb (or Newton) energy of $f$; and when $n=3$, $\mathcal{I}_2(f)$ describes the real physical energy associated with the charge density $f(x)$ (see e.g., \cite{Landkof1972,LiebLoss2001}). The second example is the case when $f=\mathbf{1}_K$ with $K$ being a convex body, i.e., $f(x)=0$ if $x\notin K$ and $f(x)=1$ if $x\in K$, then $\Ia(f)=\Ia(\mathbf{1}_K)$ becomes  the $(\alpha+1)$-chord integral of $K$ up to a multiplicative constant (see e.g.,  \cite[(3.11)]{LXYZ}). Note that the chord integral is a fundamental object and plays an important role in the integral geometry, see e.g.,  \cite{Bernig2011, Bernig2014,   Santalo2004,  Zhang1991} for some important works related to the chord integral. The $(\alpha+1)$-chord integral contains many well-studied objects in convex geometry as its special cases, including the surface area and volume of $K$. 

We define the first order variation of the Riesz $\alpha$-energy of $f$ along $g$ with respect to the Asplund sum as follows. 
\begin{definition} Let $\alpha>0$, $f\in \LC$, and $g: \Rn\to \R$ be a log-concave function.  The first order variation of the Riesz $\alpha$-energy of $f$ along $g$ with respect to the Asplund sum, denoted by $\delta \Ia(f, g)$, is defined by 
\begin{align}\label{first variation}
\delta \Ia(f,g)=\frac{1}{2}\lim_{t\to0^{+}}\frac{ \Ia(f\oplus t\bullet g)-\Ia(f)}{t}, 
\end{align} provided that the above limit exists. 
\end{definition}

In general, it is not easy to find  an explicit integral expression for  $\delta \Ia(f,g)$. However,  one can always calculate $\delta \Ia(f,f)$ precisely, see Propositions \ref{prop 4.3} and \ref{Le.ff-2}. Based on these results, we are able to get the following result. 

\begin{theorem}\label{var.bound.-1}
 Let $\alpha\!>\! 0$,  $f\!=\! e^{-\varphi}\in {\rm LC}_n$ and $g\!=\!e^{-\psi}\in {\rm LC}_n$.  If   \eqref{nat-condition-1} holds and $o\in \mathrm{int}(K_f)$, 
 then
\begin{align*} 
 \delta \Ia(f,g)
=\int_{K_f}\int_{K_f}\frac{\psi^*(\nabla\varphi(x))  f(y)f(x)}{|x-y|^{n-\alpha}}dxdy
+\int_{\partial K_f}\int_{K_f}\frac{ h_{K_g}(\nu_{K_f}(x)) f(y)f(x) }{|x-y|^{n-\alpha}}dyd\mathcal{H}^{n-1}(x).
\end{align*}
\end{theorem} Unlike those in \cite{FYZZ, HLXZ,Rot20, Rotem2022, Ulivelli}, our Theorem \ref{var.bound.-1} will be proved based on analytical tools,  instead of using the geometric variational formulas for convex bodies. When $f=e^{-\varphi}$ satisfies  that $K_f=\Rn$ (equivalently  $\dom \varphi=\{x\in \Rn: \varphi(x)<\infty\}=\Rn$), it is easier to get $\delta\Ia(f,g)$, as the boundary of unbounded convex set $K_f$ is not involved,  that is, in this case,
    \begin{align*}
        \delta \Ia(f,g)
=\int_{\mathbb{R}^n}\int_{\mathbb{R}^n}\frac{\psi^\ast(\nabla\varphi(y))f(x)f(y)}{|x-y|^{n-\alpha}}dxdy.
    \end{align*}
Indeed, using the translation invariance of $\Ia(f)$, we can  get the following result.  
\begin{corollary}\label{var.bound.}
Let $\alpha>0$,  $f=e^{-\varphi}\in {\rm LC}_n$ and $g=e^{-\psi}\in {\rm LC}_n$.  If $\varphi$ and $\psi$ satisfy \eqref{nat-condition-1}  and $\mathrm{int}(K_f)\neq \emptyset$, then  the integral formula in Theorem  \ref{var.bound.-1} still holds. 
 
\end{corollary}
 
As usual, the two integral terms in Theorem \ref{var.bound.-1} define two measures. The second integral in Theorem \ref{var.bound.-1} generates $\Rma^s(f, u)$, a Borel measure on $\sphere$ as follows: 
\begin{align*}
 \int_{\sphere}\!\! \mathbf{g}(u)d\Rma^s(f, u)\!=\!\!\! \int_{\partial K_f}\!\int_{K_f}\!\!\!\!\frac{\mathbf{g}(\nu_{K_f}\!(x)) f(y)f(x) }{|x-y|^{n-\alpha}}dyd\mathcal{H}^{n-1}(x) \! =\!\! \!\int_{\partial K_f}\!\! \!\mathbf{g}(\nu_{K_f}\!(x)) I_{\alpha}(f,x)f(x)d\mathcal{H}^{n-1}(x), 
\end{align*} for each Borel function $\mathbf{g}:\sphere \to\R$ such that $\mathbf{g}\in\mathcal{L}^1(\Rma^s(f, \cdot))$ or $\mathbf{g}$ is non-negative.

 Of more important for our paper in Section \ref{section:-6} is the Borel measure $\Rma(f,\cdot)$ defined on $\Rn$, the 
Riesz $\alpha$-energy measure of $f$ induced by the first integral in Theorem \ref{var.bound.-1}. Namely,  
\begin{align}
\int_{\mathbb{R}^n} \mathbf{h}(z)d\Rma(f,z)=\int_{K_f}\int_{K_f} \frac{\mathbf{h}(\nabla\varphi(y))f(x)f(y)}{|x-y|^{n-\alpha}}dxdy, \label{def-Riesz-measure-1}
\end{align}
for each Borel function $\mathbf{h}:\mathbb{R}^n\to\mathbb{R}$ such that $\mathbf{h}\in\mathcal{L}^1(\Rma(f,\cdot))$ or $\mathbf{h}$ is non-negative. According to \eqref{Riesz-P-def-1} and \eqref{def-Riesz-measure-1}, it follows from the Tonolli-Fubini theorem that \begin{align}
\int_{\mathbb{R}^n}\mathbf{h}(z)d\Rma(f,z)&= \int_{K_f}\left(\int_{K_f} \frac{\mathbf{h}(\nabla\varphi(y))f(x)f(y)}{|x-y|^{n-\alpha}}dx\right) dy \nonumber \\&=\int_{K_f} \left(\int_{K_f} \frac{f(x)}{|x-y|^{n-\alpha}}dx\right) \mathbf{h}(\nabla\varphi(y)) f(y) dy \nonumber\\&=\int_{K_f}  \mathbf{h}(\nabla\varphi(y)) I_{\alpha}(f,y)  f(y)dy.\label{def-Rma-22-1}
\end{align} That is, the Riesz $\alpha$-energy measure of $f$ is the push-forward of the measure $f(y)I_{\alpha}(f,y)  dy$ under $\nabla \varphi.$  A nature question to study is   the following 
 Riesz $\alpha$-energy Minkowski problem aiming to characterize the  Riesz $\alpha$-energy measure of log-concave functions. 

\begin{problem}[{\bf The Riesz $\alpha$-energy Minkowski problem}] \label{Chord-Mink-Pro-1}
 Let $\mu$ be a finite Borel measure on $\Rn$ and $\alpha>0$. Find the necessary and/or sufficient conditions on $\mu$ so that $
    \mu=\Rma(f,\cdot)$
holds for some log-concave function $f\in\LC$. 
\end{problem}

When the function $\varphi$ is  smooth enough, then 
\eqref{def-Rma-22-1} and $z=\nabla \varphi(y)$ give  \begin{align}\label{RN}
\frac{\,d\Rma(f, z)}{dz}=I_{\alpha}(f,\nabla \varphi^*(z)) f(\nabla \varphi^*(z))\det (\nabla^2\varphi^*(z)). 
\end{align} If the measure $\mu$ has a sufficiently smooth density function $\mathbf{h}$ on $\Rn$, Problem \ref{Chord-Mink-Pro-1} reduces to the following Monge-Amp\`{e}re type partial differential equation: 
\begin{align}
\mathbf{h}(z)=I_{\alpha}(f,\nabla \varphi^*(z)) f(\nabla \varphi^*(z))\det (\nabla^2\varphi^*(z)), \label{M-A-a-R-e}
\end{align} where $f=e^{-\varphi}$ is the unknown function.

The Minkowski type problems play central roles in (convex) geometry. Such problems aim to find sufficient and/or necessary conditions on a pre-given measure $\nu$ (defined on $\sphere$) so that $\nu$ is equal to a Borel measure $\mathfrak{m}(K, \cdot)$   generated by a variational formula related to some continuous functionals defined on convex bodies in terms of certain  perturbations of $K$. For example, if $\mathfrak{m}(K, \cdot)$ equals the surface area measure of $K$ (obtained by the variational formula of volume in terms of Minkowski addition), then one gets the classical Minkowski problem back to Minkowski \cite{Minkowski1987, Minkowski1903}. See e.g., \cite{GHWXY19, HLYZ2010, HLYZ16, KL2023,Lut93, LYZ18, XY17} for other extensively studied Minkowski type problems. The geometric Minkowski problem closely related to the Riesz $\alpha$-energy is the recent chord Minkowski problem, initiated by  Lutwak, Xi, Yang and Zhang in  \cite{LXYZ}, where the authors established a variational formula of chord integral and defined a new family of Borel measures: chord measures. The Minkowski type problem for chord measures has already received a lot of attention, see e.g.,  \cite{GXZ, HHL2023, HHLW, HQ,YLi24-1,YLi24-2,Qin2023,Qin2024,XYZZ}.

The study of the Minkowski problem for log-concave functions starts from the seminal work \cite{CK2015}, where Cordero-Erausquin and Klartag proposed the functional Minkowski problem aiming to characterize the surface area measure $\mu_f$ defined in \eqref{surface-a-f}. In particular, Cordero-Erausquin and Klartag proved the existence and uniqueness of solutions to the functional Minkowski problem. Please see \cite{CF13} for a more complete setting of the functional Minkowski problem. Around the same time, Rotem in \cite{Rot20} and Fang, Xing and Ye in \cite{FXY20+} provided solutions to the functional $L_p$ Minkowski problem for $p\in (0, 1)$, and respectively, for $p>1$. A first step towards the functional dual Minkowski problem was given in \cite{HLXZ} and was pushed forward to the functional dual Orlicz Minkowski problem in \cite{FYZZ}. All of the above mentioned solutions to the functional Minkowski type problems do provide weak solutions to their corresponding Monge-Amp\`{e}re equations. These further enhance the already existing close connections between the Minkowski problems and partial differential equations. 

Our second main result in this paper is the existence of solutions to the Riesz $\alpha$-energy Minkowski problem when the given measure $\mu$ is even. In particular,  Theorem \ref{Solution-a-R-e} provides a weak solution to the Monge-Amp\`{e}re equation \eqref{M-A-a-R-e}. 

\begin{theorem}\label{Solution-a-R-e}
Let $\alpha>0$. Suppose that $\mu$ is an even nonzero finite Borel measures on $\R^n$  satisfying that $\mu$ is not concentrated in any subspaces and the first moment of $\mu$ is finite. There exists a log-concave function $f\in \LC$ with $\Ia(f)=\int_{\Rn}\, \,d\mu$, such that, $
   \mu=\Rma(f,\cdot). $
\end{theorem}
 
Our paper is structured as follows. In Section \ref{section:-2}, we present some basic notations and background information that are needed. Some useful properties of both Riesz $\alpha$-potential and  Riesz $\alpha$-energy of log-concave functions are explained in Section \ref{Section:-3}. These properties include the homogeneity, translation-invariance, monotonocity, and finiteness of some integrals related to the Riesz $\alpha$-potential and Riesz $\alpha$-energy. Section \ref{nobound.} aims to find the explicit integral expression of $\delta \Ia(f,g)$ when $f=e^{-\varphi}$ and $g=e^{-\psi}$ satisfying that $\varphi^*$ and $\psi^*$ are proportional to each other (i.e., the case when $g=\tau_0\bullet f$ for some constant $\tau_0>0$). Section \ref{Section:-5} is dedicated to proving the general variational formulas for the Riesz $\alpha$-energy, i.e., Theorem  \ref{var.bound.-1}. Our main result Theorem \ref{Solution-a-R-e}, which provides a solution to the Riesz $\alpha$-energy Minkowski problem (i.e., Problem \ref{Chord-Mink-Pro-1}), is provided in Section \ref{section:-6}. 

\section{Preliminaries and notations}\label{section:-2}
We now introduce basic background for convex and log-concave functions. We refer readers to the book \cite{Roc70} for more details on the theory of convex functions, and the book \cite{Sch14} for more details on the geometric theory of convex sets.  

Let $\Rn$ be the $n$-dimensional, $n\geq 1$, Euclidean space. The letter $o$ denotes the origin of $\Rn$. Throughout, let  $\varphi: \Rn\rightarrow \R \cup\{+\infty\}$ be a convex function. 
The effective domain of $\varphi$ is $$\dom (\varphi)=\{x\in\mathbb{R}^n\,:\, \varphi(x)<+\infty\},$$
which is a convex set in $\R^n$. If $\dom (\varphi)\neq \emptyset$, then $\varphi$ is called a proper function. Hereafter, we say  a set $E\subset\Rn$ is convex if  $\lambda x+(1-\lambda)y\in E$ holds for all $x, y\in E$ and $\lambda\in [0, 1].$ The closure of $E\subset \Rn$ is denoted by $\overline{E}$. Thus, $\overline{\dom(\varphi)}$ is a closed convex set in $\Rn$. A function $f: \Rn\rightarrow [0, \infty)$ is called a log-concave function if $f=e^{-\varphi}$ for some convex function $\varphi$. If $x\in \dom(\varphi)$, then $f=e^{-\varphi}>0$. Hence the function $f=e^{-\varphi}$ is supported on $\overline{\dom(\varphi)}$. For convenience, we always let $K_f=\overline{\dom(\varphi)}$ to mean the closure of the support of a log-concave function $f$. That is, $$K_f = \overline{\{x\in \R^n: f(x)\neq 0\}}=\overline{\{x\in\mathbb{R}^n\,:\, \varphi(x)<+\infty\}}.$$ For a closed convex set $E\subset \Rn$, we always use $\partial E$ to mean the boundary of $E$, and $\nu_E$ to mean the Gauss map on $\partial E$ which maps subsets of $\partial E$ to subsets of the unit sphere $\sphere=\{x\in \Rn: |x|=1\}$ with $|x|$ being the Euclidean norm of $x\in \Rn.$ We use  $\langle x, y\rangle $ for the inner product of $x,y\in \R^n$.  We write   $B_2^n(y, R)$
 for the Euclidean ball centered at $y$ with radius $R.$

Note that for a closed convex set $E\subset \Rn$,  its Gauss map $\nu_E$ is defined almost everywhere on $\partial E$ with respect to $\mathcal{H}^{n-1}\big|_{\partial E}$, the $(n-1)$-dimensional Hausdorff measure of $\partial E.$ Throughout the paper, whenever there is no confusion on the set $F\subset \Rn$, we always write $\mathcal{H}^{n-1}\big|_F$ by $\mathcal{H}^{n-1}$. In particular, we will write $du$ for $d\mathcal{H}^{n-1}(u)\big|_{\sphere}$ for convenience. Since $K_f$ is a closed convex set in $\Rn$,   $\nu_{K_f}$  is defined almost everywhere on $\partial K_f$ with respect to $\mathcal{H}^{n-1}\big|_{K_f}$. Clearly, a log-concave function $f$ is differentiable  almost everywhere on $\R^n$ with respect to the Lebesgue measure.

By $\LC$, we mean the set of upper semi-continuous log-concave functions with nonzero finite $\mathcal{L}^1$ norm. That is, if $f\in \LC$, then $f=e^{-\varphi}$ with $\varphi$ a lower-semicontinuous convex function such that $0<\int_{\Rn} f\,dx<\infty$ where $\,dx$ denotes the Lebesgure measure on $\Rn.$ 
We will often use the following well-known and easily established fact (see e.g., \cite[Lemma 2.5]{CF13}):  a log-concave function $f=e^{-\varphi}$ has finite  $\mathcal{L}^1$ norm if and only if
 \begin{equation*}
\liminf_{|x|\rightarrow \infty} \frac{\varphi(x)}{|x|}>0.
 \end{equation*}   In particular,  there exist constants $b>0$ and $c\in \R$ such that \begin{align} \label{In.CF13-11} \varphi(x)\geq b|x|+c. \end{align}

The operations on log-concave functions considered in this paper are the Asplund sum and  the scalar multiplication.  Let $t>0$  be a constant and $f=e^{-\varphi}, g=e^{-\psi}: \Rn\rightarrow [0, \infty)$ be two log-concave functions. The  Asplund sum  of $f$ and $g$, denoted by $f\oplus g$, and the scalar multiplication of $t$ and $f$, denoted by $t \bullet f$, are defined by  
\begin{equation}\label{def-asp-sum}
	(f\oplus g) = e^{-(\varphi^*+\psi^*)^*} \ \ \ \mathrm{and} \ \
	t\bullet f = e^{-(t\varphi^*)^*},
\end{equation} where $\phi^*$ is the Legendre transform of $\phi:\R^n\to\R \cup \{+\infty\}$ defined in  \eqref{def-dual}; namely for $y\in\mathbb{R}^n$, $$\phi^*(y)=\sup_{x\in\mathbb{R}^n}\left\{\langle x,y\rangle-\phi(x)\right\}.$$  It follows from \eqref{def-dual} that, for any function $\phi:\R^n\to\R\cup \{+\infty\}$ (not necessarily convex),  $\phi^*$ is always a lower semi-continuous convex function, and $(\phi^*)^*=\phi^{**}\leq\phi$ with equality if and only if $\phi$ is a lower semi-continuous convex function.  For $c>0$ and $a\in \R$, one has
\begin{equation*}
\label{Eq.(avarphi)*} (\phi+a)^* = \phi^*-a \ \ \mathrm{and} \ \ 
	(c\phi)^* = c \phi^*(\cdot/c).
\end{equation*} If $\phi(x_0)< \infty$ for some $x_0\in \Rn$, then for any $y\in \Rn$, $\phi^*(y)\geq \langle x_0, y\rangle-\phi(x_0)>-\infty$. Moreover, \begin{align*}
\label{equ-dual-1} \phi(x)+ \phi^*(y)\geq \langle x,y\rangle \ \ \ \mathrm{for\ any}\ \ x, y\in\mathbb{R}^n.
\end{align*} It is also a direct consequence from \eqref{def-dual} that  $\phi_1^*\leq \phi_2^*$ if $\phi_1\geq \phi_2. $ By \eqref{def-asp-sum}, for log-concave functions $f=e^{-\varphi}, g=e^{-\psi}$, and $t,s>0$, 
\begin{equation} \label{addition-minkowski-relation}
	(t\bullet f\oplus s\bullet g)=e^{-(t\varphi^*+s\psi^*)^*}.
\end{equation} The one we often use is the special case $(f\oplus t\bullet g)=e^{-(\varphi^*+t\psi^*)^*}.$

The Asplund sum of log-concave functions has strong connections with the Minkowski sum of sets. To see this, we need two special functions defined on a subset $E\subset \mathbb{R}^n$: the indicator function of $E$ denoted by  $\mathbf{1}_E$ and the characteristic function of $E$ denoted by $\chi_{E}=-\log \mathbf{1}_E$. That is, $\mathbf{1}_E(x)=1$ and so $\chi_E(x)=0$ if $x\in E$;  and $\mathbf{1}_E(x)=0$ and so $\chi_E(x)=+\infty$ if $x\notin E$. We say $K\subset\Rn$ a convex body if $K$ is a convex compact set with nonempty interior. Denote by $\mathcal{K}^n$ the set of all convex bodies. Then, for $K\in \mathcal{K}^n,$ a simple calculation based on \eqref{def-dual} gives: for $y\in \Rn,$  
\begin{align}
    (\chi_K)^*(y)=h_K(y):=\sup\{\langle x, y\rangle: x\in K\}. \label{relation-supp-dual}
\end{align} Hereafter the function $h_K: \Rn\rightarrow \R$ is called the support function of $K.$ Let us pause here to mention that the support function can be defined for general convex sets (for example unbounded  closed convex compact sets, or convex compact sets with empty interiors). Clearly, $h_{tK}=th_K$ if $t>0$.  Now if  $f=\mathbf{1}_K$ and $g=\mathbf{1}_L$ for $K, L\in \mathcal{K}^n,$ then  \eqref{addition-minkowski-relation}  and  \eqref{relation-supp-dual}  imply that  \begin{equation*}  
	(t\bullet f\oplus s\bullet g)=e^{-(t\chi_K^*+s\chi_L^*)^*}=e^{-(th_K+sh_L)^*}=e^{-(h_{tK+sL})^*}=e^{-\chi_{tK+sL}}=\mathbf{1}_{tK+sL},
\end{equation*} where $tK+sL=\{tx+sy: x\in K, y\in L\}$ is the Minkowski addition of $K$ and $L$ with coefficients $t, s>0$. Note that $h_{tK+sL}=th_K+sh_L.$ In other words, the Asplund sum reduces to the Minkowski sum of convex bodies when $f$ and $g$ are the indicator functions of convex bodies. We would like to mention that 
$\varphi^*$ for a convex function $\varphi$ resembles 
the support function $h_K$ for convex body, and  indeed  $\varphi^*$  is  the support function of $f=e^{-\varphi}$ and often denoted by $h_f$ (i.e., $h_f=\varphi^*$), see e.g., \cite{AM11,Roc70} for more details. 

For $E\subset \Rn$, its interior is denoted by $\mathrm{int}(E)$. Denote by $\nabla \varphi$ the gradient of $\varphi$ whenever it exists. A well-known fact is that, if $\varphi: \Rn \rightarrow \R\cup\{+\infty\}$ is a convex function, then $\varphi$ is differentiable almost everywhere in $\mathrm {int}(\dom(\phi))$ with respect to Lebesgue measure. 
If $\nabla \varphi(x)$ exists, then
\begin{equation}
	\label{basicequ}
	\varphi(x)+\varphi^*(\nabla\varphi(x)) = \langle x, \nabla \varphi(x)\rangle.
\end{equation}
We will need  the following lemma proved in \cite[Proposition 2.1]{Rot20}. 
\begin{lemma}\label{rotem-111}
Let $\varphi, \mathfrak{g}:\mathbb{R}^n\to\mathbb{R}\cup\{+\infty\}$ be lower semi-continuous functions. Assume that $\mathfrak{g}$ is bounded from below and that $\varphi(o),\mathfrak{g}(o)<\infty$. Then at every point $x\in\mathbb{R}^n$ where $\varphi^*$ is differentiable, we have
\begin{eqnarray*}\label{Rot-000}
\frac{d}{dt}\Big|_{t=0^+}(\varphi+t\mathfrak{g})^*(x)=-\mathfrak{g}(\nabla \varphi^*(x)) .
\end{eqnarray*}
\end{lemma}

\section{The Riesz \texorpdfstring{$\alpha$}{}-potential and  Riesz \texorpdfstring{$\alpha$}{}-energy of log-concave functions}	\label{Section:-3}

In this section, we will study basic properties of Riesz $\alpha$-potential and Riesz $\alpha$-energy of log-concave functions.  Let $f\in \LC$  and $\alpha>0$. Recall that the Riesz $\alpha$-potential of log-concave function $f$ defined in \eqref{Riesz-P-def-1} is: for $x\in \Rn$,  
\begin{align*}
I_{\alpha}(f, x)= \int_{\mathbb{R}^n}\frac{f(y)}{|x-y|^{n-\alpha}}\,dy.\end{align*} The Riesz $\alpha$-energy of log-concave function $f$ is formulated by 
\begin{align} \label{alpha-energy-def}
\Ia(f)=\int_{\mathbb{R}^n}\int_{\mathbb{R}^n}\frac{ f(x) f(y)}{|x-y|^{n-\alpha}}dxdy=\int_{\mathbb{R}^n}    f(y) I_{\alpha}(f, y)dy.\end{align} 
When $K\subset\R^n$ is a convex body, one sees that \begin{align*}
\Ia(\mathbf{1}_K)=\int_{K}\int_{ K}\frac{ 1}{|x-y|^{n-\alpha}}dxdy,\end{align*} which equals the $(\alpha+1)$-chord integral of $K$ up to a multiplicative constant (see e.g.,  \cite[(3.11)]{LXYZ}). 

The Riesz $\alpha$-energy has the following properties. 
\begin{proposition}\label{properties-a-energy}
    Let $\alpha>0$ and $f\in \LC$. 
    
  \vskip 2mm \noindent i)
  For any $c>0$, one has  \begin{align}\label{homogeneity}
\Ia(cf)=c^2 \Ia(f) \ \ \ \mathrm{and}\ \ \Ia(f\circ c)=c^{-(n+\alpha)} \Ia(f),
\end{align}  where $(f\circ c)(x)=f(cx)$ for $x\in \Rn.$

\vskip 2mm \noindent ii)  For any $x_0\in \Rn,$\begin{align}\label{translation invariance}
\Ia( f(\cdot+x_0))=  \Ia(f).
\end{align}

\vskip 2mm \noindent iii) For $g\in \LC$ such that $f\leq g$, then 
\begin{align}\label{monotone increasing}
\Ia(f)\leq   \Ia(g).
\end{align}

\end{proposition}

\begin{proof}
{\it i)}:  The first equality in \eqref{homogeneity} is trivial, so we only show the second equality. It follows from \eqref{alpha-energy-def}, $\tilde{x}=cx$ and $\tilde{y}=cy$ that 
\begin{align*}
\Ia(f\circ c)=\int_{\mathbb{R}^n}\int_{\mathbb{R}^n}\frac{ f(cx) f(cy)}{|x-y|^{n-\alpha}}dxdy=c^{-(n+\alpha)}\int_{\mathbb{R}^n}\int_{\mathbb{R}^n}\frac{ f(\tilde{x}) f(\tilde{y})}{|\tilde{x}-\tilde{y}|^{n-\alpha}}d\tilde{x}d\tilde{y}=c^{-(n+\alpha)}\Ia(f).\end{align*}

\noindent {\it ii)}:  For any $x_0\in\R^n$, one has, by the variable changes $\tilde{x}=x+x_0$ and  $\tilde{y}=y+x_0$, 
\begin{align*}
\Ia(f(\cdot +x_0))=\int_{\mathbb{R}^n}\int_{\mathbb{R}^n}\frac{ f(x+x_0) f(y+x_0)}{|x-y|^{n-\alpha}}dxdy= \int_{\mathbb{R}^n}\int_{\mathbb{R}^n}\frac{ f(\tilde{x}) f(\tilde{y})}{|\tilde{x}-\tilde{y}|^{n-\alpha}}d\tilde{x} d\tilde{y}=\Ia(f).\end{align*} 

\noindent {\it iii)}:  As $f\leq g$, one gets $f(x)f(y)\leq g(x)g(y)$ for all $x, y\in \Rn.$ The desired inequality \eqref{monotone increasing} is an immediate consequence of \eqref{alpha-energy-def}. 
\end{proof}

We have the following proposition. 

\begin{proposition} \label{estimate-riesz potential} 
    Let $f\!=\!e^{-\varphi}\!\in\! \LC$ and $\alpha\! >\!0$. For $y\in\Rn$ fixed,  one gets $0\!<\!I_{\alpha}(f,y)\!<\!\infty$ and 
    \begin{align}   -\infty< \int_{\mathbb{R}^n}\frac{\varphi(x)f(x)}{|x-y|^{n-\alpha}}\,dx<\infty. \label{positive-r-potential}
    \end{align}
\end{proposition}
\begin{proof} Let  $f=e^{-\varphi}\in \LC$. As $0<\int_{\Rn} e^{-\varphi}\,dx<\infty,$   $\{x: f(x)\neq 0\}$ is a set with positive Lebesgue measure, and also \eqref{In.CF13-11} holds.  The former one yields $ I_{\alpha} (f,y)>0$ for any $y\in \Rn.$ The latter one implies that,  for any $x\in\mathbb{R}^n$, there exist constans $b>0$ and $c\in\R$ such that
\begin{eqnarray}\label{upperboundoff0}
f(x)\leq e^{-c}e^{-b|x|}. 
\end{eqnarray}    
Taking $z=x-y$ and using  the triangle inequality, one gets
\begin{align*}
    I_{\alpha}(f,y)&\leq   e^{-c}\int_{\R^n}\frac{e^{-b|x|}}{|x-y|^{n-\alpha}}dx\\
    &\leq  e^{-c}\int_{\R^n}e^{-b\big||z|-|y|\big|}|z|^{\alpha-n}dz\\
    &= e^{-c}e^{b|y|}\int_{\{z\in\mathbb{R}^n: |y|<|z|\}}e^{-b|z|}|z|^{\alpha-n}dz+ e^{-c}\int_{\{z\in\mathbb{R}^n: |y|\geq|z|\}}e^{b|z|-b|y|}|z|^{\alpha-n}dz.
\end{align*} From the polar coordinates formula,  it is easy to check that 
\begin{align*}
 \int_{\{z\in\mathbb{R}^n: |y|\geq|z|\}}e^{b|z|-b|y|}|z|^{\alpha-n}dz&\leq \int_{\{z\in\mathbb{R}^n: |y|\geq|z|\}} |z|^{\alpha-n}dz= \frac{n\omega_n|y|^{\alpha}}{\alpha},
\end{align*} where $\omega_n$ denotes the volume of the unit ball.  Similarly, 
\begin{align*}
 e^{b|y|}\int_{\{z\in\mathbb{R}^n: |y|<|z|\}}e^{-b|z|}|z|^{\alpha-n}dz
 &\leq e^{b |y|}\int_{\R^n}e^{-b|z|}|z|^{\alpha-n}dz=b^{-\alpha}n\omega_n \Gamma (\alpha)e^{b|y|},
\end{align*} 
 where $\Gamma(\cdot)$ denotes the Gamma function: $\Gamma(s)=\int_0^{\infty} e^{-t} t^{s-1}\,dt$ for any $s>0$. 
Therefore, for any fixed $y\in\R^n$, we have
\begin{align*}
    I_{\alpha}(f,y)& \leq \frac{e^{-c}n\omega_n}{\alpha}|y|^{\alpha} +e^{-c}b^{-\alpha}n\omega_n \Gamma (\alpha)e^{b|y|}, 
\end{align*} which shows $0<I_\alpha (f,y)<\infty$ for any $f\in{\rm LC}_n$ and $y\in \Rn$ fixed.

Now let us prove \eqref{positive-r-potential}. To this end, by \eqref{In.CF13-11},    $\varphi-c\geq b|x|\geq 0$ for all $x\in \Rn.$  Applying the fact that $e^{-\frac{s}{2}}s\leq 2e^{-1}$ for $s\ge 0$ to $\varphi-c$, one gets $e^{-\frac{\varphi-c}{2}}(\varphi-c)\leq 2e^{-1}.$ This further gives $$\varphi e^{-\varphi}\leq 2e^{-1-\frac{c}{2}}e^{-\frac{\varphi}{2}}+ce^{-\varphi}.$$ Note that $e^{-\frac{\varphi}{2}}\in \LC$ and hence for all $y\in \Rn,$
$$
\int_{\mathbb{R}^n}\frac{\varphi(x)e^{-\varphi(x)}}{|x-y|^{n-\alpha}}dx\leq  2e^{-1-\frac{c}{2}} \int_{\mathbb{R}^n}\frac{e^{-\frac{\varphi(x)}{2}}}{|x-y|^{n-\alpha}}dx+ c\int_{\mathbb{R}^n}\frac{e^{-\varphi(x)}}{|x-y|^{n-\alpha}}dx<\infty.
$$ On the other hand, as $\varphi\geq c,$ one sees that for all $y\in \Rn$, $$
\int_{\mathbb{R}^n}\frac{\varphi(x)e^{-\varphi(x)}}{|x-y|^{n-\alpha}}dx\geq c 
\int_{\mathbb{R}^n}\frac{e^{-\varphi(x)}}{|x-y|^{n-\alpha}}dx>-\infty.$$ This completes the proof of \eqref{positive-r-potential}.  \end{proof}

Our next result is regarding the finiteness of the Riesz $\alpha$-energy of log-concave functions.

\begin{proposition}\label{finiteness of I_q}
For $f\in \LC$ and $\alpha >0$, one has $0< \Ia(f)<\infty.$ 
\end{proposition}
\begin{proof} Let  $f=e^{-\varphi}\in \LC$. As $0<\int_{\Rn} e^{-\varphi}\,dx<\infty,$ one sees that $\{x: f(x)\neq 0\}$ is a set with positive Lebesgue measure, and also \eqref{upperboundoff0} holds. The former one yields $ \Ia(f)>0.$   
By \eqref{upperboundoff0}  and Proposition \ref{properties-a-energy}, one gets $$\Ia(f)\leq\Ia(e^{-b|\cdot |-c})= e^{-2c} b^{-n-\alpha} \Ia(e^{-|\cdot|}).  $$ Hence, if $\Ia(e^{-|\cdot|})<\infty$, the desired inequality $\Ia(f)<\infty$ follows,

Taking $z=x-y$ and using  the triangle inequality, one gets 
\begin{align}
\Ia(e^{-|\cdot|})& = \int_{\mathbb{R}^n}\int_{\mathbb{R}^n}\frac{ e^{-|x|}e^{-|y|} }{|x-y|^{n-\alpha}}dxdy \nonumber \\
 &\leq  \int_{\mathbb{R}^n}\int_{\mathbb{R}^n} e^{- \big||z|-|y|\big|}e^{- |y| } |z|^{\alpha-n}dzdy \nonumber \\
 &= \int_{ \mathbb{R}^n}\int_{\{z\in\mathbb{R}^n: |y|<|z|\}} e^{-  |z|  }  |z|^{\alpha-n}dzdy +\int_{ \mathbb{R}^n}\int_{\{z\in\mathbb{R}^n: |y|\geq|z|\}} e^{ |z|  -2 |y| } |z|^{\alpha -n}dzdy. \label{estimate-1}
  \end{align}
 From  the Fubini's theorem and the polar coordinates formula,  it is not hard to check that  
  \begin{align*}
  A_1&:=\int_{ \mathbb{R}^n}\int_{\{y\in\mathbb{R}^n: |y|<|z|\}} e^{-  |z|  }  |z|^{\alpha-n}dydz \\
  & =\int_{ \mathbb{R}^n} e^{-|z| }  |z|^{\alpha -n}\left(\int_{\{y\in\mathbb{R}^n: |y|<|z|\}}dy\right)dz\\
 & =\omega_n\int_{ \mathbb{R}^n}e^{- |z| } |z|^{\alpha }dz\\& = n\omega_n^2\Gamma(n+\alpha )<\infty.
  \end{align*}  
 Similarly,  one has 
\begin{align*}
 \int_{ \mathbb{R}^n}\int_{\{y\in\mathbb{R}^n: |z|\leq |y|\}}  e^{- |y| } |z|^{\alpha-n}dydz
 &=\int_{ \mathbb{R}^n} |z|^{\alpha-n}\left(n\omega_n\int_{|z|}^{+\infty}  e^{-t}t^{n-1}dt\right)dz\\
 &\leq\int_{ \mathbb{R}^n} |z|^{\alpha-n} e^{-\frac{|z|}{2}}\left(n\omega_n\int_{|z|}^{+\infty}  e^{-\frac{ t}{2}}t^{n-1}dt\right)dz\\
 &\leq\int_{ \mathbb{R}^n} |z|^{\alpha-n} e^{-\frac{|z|}{2}}\left(n\omega_n\int_{0}^{+\infty}  e^{-\frac{t}{2}}t^{n-1}dt\right)dz\\
 &= n\omega_n 2^n \Gamma(n)\int_{ \mathbb{R}^n} |z|^{\alpha-n} e^{-\frac{|z|}{2}} dz\\
 &= (n\omega_n)^2 2^{\alpha+n}\Gamma(n)\Gamma(\alpha)<\infty.
  \end{align*}
 By the Fubini's theorem, one gets 
\begin{align*}
 A_2:=\int_{ \mathbb{R}^n}\int_{\{z\in\mathbb{R}^n: |y|\geq|z|\}} e^{ |z|  -2 |y| } |z|^{\alpha -n}dzdy  &
 \leq\int_{ \mathbb{R}^n}\int_{\{z\in\mathbb{R}^n: |z|\leq |y|\}}  e^{-|y| } |z|^{\alpha-n}dzdy\\   &=\int_{ \mathbb{R}^n}\int_{\{y\in\mathbb{R}^n: |z|\leq |y|\}}  e^{- |y| } |z|^{\alpha-n}dydz<\infty. \end{align*} It follows from \eqref{estimate-1} that  $\Ia(f)$ is finite, and this concludes the proof. 
\end{proof}
 
It has been proved in \cite[Theorem 5.12]{HLXZ} (see also  \cite[Proposition 3.2]{FYZZ}) that, for $f=e^{-\varphi}\in \LC$ and $\alpha>0$, one has \begin{eqnarray}\label{FYZZ-1--1}\int_{\mathbb{R}^n} \frac{|\nabla f(y)|}{  |y|^{n-\alpha}} dy=
\int_{\mathbb{R}^n} \frac{|\nabla \varphi(y)|e^{-\varphi(y)}}{    |y|^{n-\alpha}} dy<\infty. 
\end{eqnarray}  In particular, when $\alpha=n, $ one gets \cite[Lemma 4]{CK2015} \begin{eqnarray}\label{FYZZ-1--2}\int_{\mathbb{R}^n} |\nabla f(y)|     dy=
\int_{\mathbb{R}^n} |\nabla \varphi(y)|  e^{-\varphi(y)} dy<\infty. 
\end{eqnarray}

 Using this inequality, we can obtain the following result.   

\begin{proposition}\label{bounded}
 Let $\alpha>0$ and $f\in \LC$.   Then 
  \begin{align*}
 \int_{\mathbb{R}^n}\int_{\mathbb{R}^n} \frac{|\nabla f(y)|f(z)}{  |z-y|^{n-\alpha}}   \,dy \,dz<\infty.
   \end{align*}
   
\end{proposition}

 \begin{proof} First, let $\alpha\geq n.$ From the triangle inequality, one has
 \begin{align*}
 \int_{\mathbb{R}^n}\int_{\{y\in\mathbb{R}^n: |y|\leq |z|\}} \frac{|\nabla f(y)| f(z)}{ |z-y|^{n-\alpha} } \,dy \,dz  
 & \leq\int_{\mathbb{R}^n}\int_{\{y\in\mathbb{R}^n: |y|\leq |z|\}} \frac{|\nabla f(y)| f(z)} {(|z|+|y|)^{n-\alpha}}  dy dz\\
 & \leq 2^{\alpha-n}\int_{\mathbb{R}^n}\int_{\{y\in\mathbb{R}^n: |y|\leq |z|\}} \frac{|\nabla f(y)|f(z)}{  |z|^{n-\alpha}}   dy dz\\
 & \leq 2^{\alpha-n}\left(\int_{   \mathbb{R}^n} |\nabla f(y)| dy\right) \times\left(\int_{\mathbb{R}^n}  \frac{f(z)}{|z|^{n-\alpha} }dz\right)<\infty,
   \end{align*} 
   where the last inequality follows from \eqref{FYZZ-1--2} 
and the fact that $f\in \LC$ which guarantees  (see \cite[Proposition 2.1]{HLXZ} or \cite[Lemma 3.1]{FYZZ}), for $\alpha>0$,  $$ \int_{\mathbb{R}^n}|z|^{\alpha-n}   f(z)dz <\infty. $$ Similarly,  when combined  with \eqref{FYZZ-1--1}, we also have
  \begin{align*}
 \int_{\mathbb{R}^n}\int_{\{y\in\mathbb{R}^n: |y| > |z|\}} \frac{|\nabla f(y)|f(z)}{ |z-y|^{n-\alpha}}   dy dz 
   & \leq 2^{\alpha-n}\int_{\mathbb{R}^n}\int_{\{y\in\mathbb{R}^n: |y|> |z|\}} \frac{|\nabla f(y)|f(z)}{ |y|^{n-\alpha }}  dy dz\\
&    \leq 2^{\alpha-n}\left(\int_{   \mathbb{R}^n} |\nabla f(y)| \cdot |y|^{\alpha-n}dy\right) \times\left(\int_{\mathbb{R}^n}    f(z)dz\right)<\infty.
   \end{align*}
 Consequently, the desired result holds when $\alpha \geq n$.

 For the case $0<\alpha<n$,  let $\epsilon_0>0$ be a fixed constant and then   \begin{align}
 \int_{\mathbb{R}^n} \left(\int_{\mathbb{R}^n} \frac{|\nabla f(y)|f(z)} { |z-y|^{n-\alpha}   } dz\right) dy  
 & =\int_{\Rn}\left( \int_{\{z\in \Rn: |z-y|\geq \epsilon_0\}} \frac{|\nabla f(y)|f(z)} { |z-y|^{n-\alpha}   }  dz \right) dy \nonumber \\ &\quad \quad \quad +\int_{\Rn}\left( \int_{\{z\in \Rn: |z-y|< \epsilon_0\}} \frac{|\nabla f(y)|f(z)} { |z-y|^{n-\alpha}   }  dz \right) dy\nonumber  \\ & \leq \epsilon_0^{\alpha-n}  \int_{\Rn} \left(\int_{\{z\in \Rn: |z-y|\geq \epsilon_0\}}  |\nabla f(y)|   f(z) dz \right) dy\nonumber \\ &\quad \quad \quad + e^{-c} \int_{\Rn}\left( \int_{\{z\in \Rn: |z-y|< \epsilon_0\}}  \frac{|\nabla f(y)|} { |z-y|^{n-\alpha}   }    dz \right) dy, \label{long-ineq-1} \end{align}  where, in the last equality, we have used \eqref{upperboundoff0}.  It can be checked that \begin{align} \int_{\Rn} \left(\int_{\{z\in \Rn: |z-y|\geq \epsilon_0\}}  |\nabla f(y)|   f(z) dz\right) dy &\leq  \int_{\Rn} \left(\int_{\Rn} |\nabla f(y)|   f(z) dz\right) dy\nonumber  \\&= \left(\int _{\Rn} |\nabla f(y)| \,dy\right) \times \left(\int_{\Rn}   f(z) dz\right)<\infty, \label{long-ineq-2} \end{align} where we have used \eqref{FYZZ-1--2} and the fact that $f\in \LC$ (hence is integrable). On the other hand, \begin{align*}
     \int_{\Rn} \left(\int_{\{z\in \Rn: |z-y|< \epsilon_0\}}  \frac{|\nabla f(y)| }{ |z-y|^{n-\alpha}}   dz\right) dy &=\int_{\Rn} |\nabla f(y)| \left(\int_{\{z\in \Rn: |z-y|< \epsilon_0\}}  |z-y|^{\alpha-n}   dz\right) dy\\ &=\int_{\Rn} |\nabla f(y)| \left(\int_{\{x\in \Rn: |x|<\epsilon_0\}}  |x|^{\alpha-n}   dx\right) dy\\ &=\frac{n\omega_n \epsilon_0^{\alpha}}{\alpha} \int_{\Rn}|\nabla f(y)| \,dy<\infty, 
 \end{align*} where again we have used \eqref{FYZZ-1--2}. Together with \eqref{long-ineq-1} and \eqref{long-ineq-2}, one gets, for $0<\alpha<n$,   \begin{align*}
\int_{\mathbb{R}^n}\left(\int_{\mathbb{R}^n} \frac{|\nabla f(y)|f(z)} {|z-y|^{\alpha-n}} dz \right) dy<\infty.\end{align*} The desired result holds for $0<\alpha <n$ by Fubini-Tonelli theorem. This completes the proof. \end{proof}  

  We will need the following result. 
  \begin{proposition}\label{finite-22-33}
Let $\alpha>0$ and $f=e^{-\varphi}\in\LC$.  Then
\begin{eqnarray}\label{ff0}
 \int_{\mathbb{R}^n}\int_{\mathbb{R}^n}   \frac{\varphi(x)f(x)f(y)}{ |x-y|^{n-\alpha}}dxdy\in(-\infty,\infty).
\end{eqnarray}
\end{proposition}
\begin{proof} Let $f\!\!=\!\!e^{-\varphi}\!\in\! \LC$. By    \eqref{In.CF13-11} (i.e., $\varphi(x)\!\geq\! c$ for any $x\!\in\!\R^n$) and   Proposition  \ref{finiteness of I_q}, one gets    
\begin{eqnarray*}
 \int_{\mathbb{R}^n}\int_{\mathbb{R}^n}   \frac{\varphi(x)f(x)f(y)}{ |x-y|^{n-\alpha}}dxdy\geq  c\Ia (f)>-\infty.
\end{eqnarray*}

To prove the boundedness from above, we first let $c\geq 0$. That is, $\varphi (x)\geq 0$ for all $x\in \Rn$ and hence $\varphi\geq \frac{\varphi}{2}.$ From the basic inequality $te^{-\frac{t}{2}}\leq \frac{2}{e}$ for $t>0$ and Proposition \ref{finiteness of I_q}, we have 
\begin{align*}
    \int_{\mathbb{R}^n}\int_{\mathbb{R}^n}   \frac{\varphi(x)f(x)f(y)}{ |x-y|^{n-\alpha}}dxdy
    \leq  \frac{2}{e}\int_{\mathbb{R}^n}\int_{\mathbb{R}^n}   \frac{ e^{-\frac{\varphi(x)}{2}}e^{-\frac{\varphi(y)}{2}}}{ |x-y|^{n-\alpha}}dxdy
    = \frac{2}{e}\Ia(e^{-\frac{\varphi}{2}}) <\infty.
\end{align*} If $c<0,$ then let $\widetilde{\varphi}=\varphi-c\geq0$ and $\widetilde{f}=e^{-\widetilde{\varphi}}=e^c f$. This further implies that  
 \begin{align*}
    \int_{\mathbb{R}^n}\int_{\mathbb{R}^n}   \frac{\varphi(x)f(x)f(y)}{ |x-y|^{n-\alpha}}dxdy& = e^{-2c} \int_{\mathbb{R}^n}\int_{\mathbb{R}^n}   \frac{\widetilde{\varphi}(x)\widetilde{f}(x)\widetilde{f}(y)}{ |x-y|^{n-\alpha}}dxdy+c\Ia(f)\\ &\leq 2 e^{-2c-1} \Ia(e^{-\frac{\widetilde{\varphi}}{2}}) +c\Ia(f)<\infty.
\end{align*}
Therefore, we obtain \eqref{ff0}.
\end{proof}

\section{Variational formulas for the Riesz \texorpdfstring{$\alpha$}{}-energy  when  \texorpdfstring{$\psi^*=\beta_1\varphi^*+\beta_2$}{}}\label{nobound.}
 Recall that,  $\delta \Ia(f,g)$ defined in \eqref{first variation}: for $\alpha>0$, $f\in \LC$, and $g: \Rn\to \R$ a log-concave function,  \begin{align*} \delta \Ia(f,g)=\frac{1}{2}\lim_{t\to0^{+}}\frac{ \Ia(f_t)-\Ia(f)}{t}, 
\end{align*} provided that the above limit exists. Hereafter, we let $f_t=f\oplus t\bullet g$. To find the integral expression for $\delta \Ia(f, g)$ is not easy; this can be seen from the variational formulas regarding log-concave functions  in \cite{CF13, FYZZ, HLXZ,Rot20, Rotem2022, Ulivelli}. Moreover, in general  $\delta \Ia(f, g)$ may  be infinite. So when calculate $\delta \Ia(f, g)$, certain conditions on the relation between $f$ and $g$ are required. Throughout the section, we always let $f=e^{-\varphi}$ and $g=e^{-\psi}.$

Before we proceed to establish the explicit integral expression of $\delta \Ia(f, g)$ for $f\in \LC$ and general $g$, let us mention the following often used but  easily established fact: if $\widetilde{\varphi}=\varphi-c$ and $\widetilde{\psi}=\psi-d$ for $c,d\in\R$, let $\widetilde{f}=e^{-\widetilde{\varphi}}, \widetilde{g}=e^{-\widetilde{\psi}}$, and then  
\begin{eqnarray*} \widetilde{f}_t=\widetilde{f}\oplus t\bullet \widetilde{g}=e^{(c+td)} (f\oplus t\bullet g) = e^{(c+td)} f_t.
\end{eqnarray*}
Proposition \ref{properties-a-energy} implies  that
$
\Ia(\widetilde{f}_t)=e^{ 2(c+td)} \Ia(f_t).
$ By the product rule, \eqref{first variation}, and Proposition \ref{properties-a-energy}, one gets, 
 \begin{align*} 
\delta \Ia(\widetilde{f},\widetilde{g})
&=\frac{e^{2c}}{2}\lim_{t\to0^{+}}\frac{e^{2td} -1 }{t} \cdot \lim_{t\to 0^+}  \Ia(f_t)  +\frac{e^{2c}}{2}\lim_{t\to0^{+}}\frac{ \Ia(f_t) - \Ia(f) }{t} \nonumber\\
&=e^{2c}\Big( \delta \Ia(f,g)+d \lim_{t\to 0^+}  \Ia(f_t) \Big).
\end{align*} In general, one cannot expect to have $\lim_{t\to 0^+} \Ia(f_t) =\Ia(f)$; but if it is the case, then 
  \begin{align}\label{add-1}
\delta \Ia(\widetilde{f},\widetilde{g})=e^{2c}\Big( \delta \Ia(f,g)+d  \Ia(f) \Big).
\end{align}

\subsection{The calculation of \texorpdfstring{$\delta \Ia(f, f)$}{}} \label{proportional case}

For the Riesz $\alpha$-energy, we can calculate $\delta \Ia(f, f)$ without any additional conditions.

\begin{proposition}\label{prop 4.3}
Let $\alpha>0$ and $f=e^{-\varphi}\in\LC$.  Then
\begin{eqnarray}\label{ff}
\delta \Ia(f,f)=\Big(\frac{n+\alpha}{2}\Big) \Ia(f)-\int_{\mathbb{R}^n}\int_{\mathbb{R}^n}   \frac{\varphi(x)f(x)f(y)}{ |x-y|^{n-\alpha}}dxdy.
\end{eqnarray} In particular, $\delta \Ia(f, f)\in (-\infty, \infty)$. 
\end{proposition}
\begin{proof}   As $f=e^{-\varphi}\in \LC$, it follows from \eqref{In.CF13-11} that $\varphi(x)\geq b|x|+c$. First, assumed  that  $c\geq0$. From \eqref{def-asp-sum}, one has   $(f\oplus t\bullet f)(x)=e^{-(1+t)\varphi\left(\frac{x}{1+t}\right)}.$ In this case,  replacing $\frac{x}{1+t}$ and $\frac{y}{1+t}$ by $x$ and $y$, respectively, one has 
\begin{align*}
    \Ia(f\oplus t\bullet f)
    &=(1+t)^{n+\alpha} \int_{\mathbb{R}^n}\int_{\mathbb{R}^n} \frac{ e^{-(1+t)\varphi(x)} e^{-(1+t)\varphi(y)} }{ |x-y|^{n-\alpha}}dxdy.
\end{align*} Hence, one can write $
2\delta \Ia(f,f)=B_1+B_2$, where $B_1$ and $B_2$ are 
 \begin{align*} 
B_1 &=\lim_{t\to0^{+}}\frac{ (1+t)^{n+\alpha}-1}{t}\int_{\mathbb{R}^n}\int_{\mathbb{R}^n} \frac{e^{-(1+t)\varphi\left(x\right)} e^{-(1+t)\varphi\left(y\right)}} {|x-y|^{n-\alpha}}dxdy,\\
B_2 &=\lim_{t\to0^{+}}\int_{\mathbb{R}^n}\int_{\mathbb{R}^n}  \frac{ e^{-(1+t)\varphi\left(x\right)} e^{-(1+t)\varphi\left(y\right)} - e^{-\varphi(x)}e^{-\varphi(y)} }{t|x-y|^{n-\alpha}}dxdy.
\end{align*} As $\varphi\geq 0$, one gets $(1+t)\varphi$ is increasing on $t$, and hence the   monotone convergence theorem can be applied to obtain 
\begin{align*}
 \lim_{t\to0^{+}} \int_{\mathbb{R}^n}\int_{\mathbb{R}^n} \frac{e^{-(1+t)\varphi\left(x\right)} e^{-(1+t)\varphi\left(y\right)}} {|x-y|^{n-\alpha}}dxdy 
= \int_{\mathbb{R}^n}\int_{\mathbb{R}^n}\lim_{t\to0^{+}} \frac{e^{-(1+t)\varphi\left(x\right)} e^{-(1+t)\varphi\left(y\right)}} {|x-y|^{n-\alpha}}dxdy= \Ia(f).
\end{align*} This in particular yields $\lim_{t\to 0^+} \Ia(f\oplus t\bullet f)=\Ia(f).$ Moreover, 
\begin{align}
B_1
=\left(\lim_{t\to 0^{+}}\frac{ (1+t)^{n+\alpha}-1}{t} \right) \cdot \left(\lim_{t\to0^{+}} \int_{\mathbb{R}^n}\int_{\mathbb{R}^n} \frac{e^{-(1+t)\varphi (x)} e^{-(1+t)\varphi(y)}}{ |x-y|^{n-\alpha} }dxdy \right)=(n+\alpha)  \Ia(f).\label{estimate-B1}
\end{align}

From the basic inequality $e^s\geq  1+s$ for $s\geq 0$,  it is easily checked that $\frac{1-e^{-\beta t}}{t}$ is   decreasing  on $t>0$ when $\beta\geq 0$, due to the fact that  $$\frac{d}{dt}\Big(\frac{1-e^{-\beta t}}{t}\Big)=\frac{e^{-\beta t}(\beta t+1)-1}{t^2}\leq0.$$   This fact can be applied to get, due to $\varphi\geq 0,$ that for any  $x,y\in\R^n$ $$0\leq \frac{1-e^{- t(\varphi(x)+\varphi(y))}}{t} $$ is decreasing on $t>0$. Moreover, for any $x, y\in \Rn$, one gets \begin{align*}  0\leq \frac{1-e^{- t(\varphi(x)+\varphi(y))}}{t}\leq \lim_{t\to 0^+} \frac{1-e^{- t(\varphi(x)+\varphi(y))}}{t}= \varphi(x)+\varphi(y),\end{align*}
where the value of $\varphi(x)+\varphi(y)$   could be $+\infty$.  
The monotone   convergence theorem, the Fubini's theorem, and \eqref{ff0} immediately yield  
 \begin{align*}
 B_2 
&=\int_{\mathbb{R}^n}\int_{\mathbb{R}^n} \lim_{t\to0^{+}}  \frac{ e^{-(1+t)\varphi\left(x\right)} e^{-(1+t)\varphi\left(y\right)} - e^{-\varphi(x)}e^{-\varphi(y)} }{t|x-y|^{n-\alpha}}dxdy \\&=\int_{\mathbb{R}^n}\int_{\mathbb{R}^n} \lim_{t\to0^{+}}  \frac{ e^{-t(\varphi\left(x\right)+\varphi\left(y\right))} - 1}{t} \cdot \frac{e^{-\varphi(x)}e^{-\varphi(y)}}{|x-y|^{n-\alpha}}   dxdy \\
&=-\int_{\mathbb{R}^n}\int_{\mathbb{R}^n} \frac{\varphi(x)+\varphi(y)}{ |x-y|^{n-\alpha}} f(x)f(y)dxdy\\
&=-2\int_{\mathbb{R}^n}\int_{\mathbb{R}^n}  \frac{\varphi(x)f(x)f(y)}{ |x-y|^{n-\alpha}} dxdy.
\end{align*} This, together with \eqref{estimate-B1}, implies that \begin{align*}
    \delta \Ia(f, f)=\frac{B_1+B_2}{2} =\Big(\frac{n+\alpha}{2}\Big) \Ia(f)-\int_{\mathbb{R}^n}\int_{\mathbb{R}^n}  \frac{\varphi(x)f(x)f(y)}{ |x-y|^{n-\alpha}}dxdy.
\end{align*} This concludes the proof when $\varphi\geq 0.$  

For $c<0$, let $\widetilde{\varphi}=\varphi-c\geq0$ and $\widetilde{f}\!=\!e^{-\widetilde{\varphi}}=e^c f$. Then, $ \widetilde{f}\oplus t\bullet \widetilde{f}\!=\! e^{(1+t)c} (f\oplus t\bullet f)$ and by \eqref{homogeneity}, $$\Ia(f\oplus t\bullet f)=e^{-2(1+t)c}  \Ia(\widetilde{f}\oplus t\bullet \widetilde{f}). $$  By the product rule, \eqref{homogeneity}, \eqref{first variation} and the fact that $\lim_{t\to 0^+} \Ia(f\oplus t\bullet f)=\Ia(f),$ one gets, $$\delta \Ia(f, f)=e^{-2c} \delta \Ia(\widetilde{f}, \widetilde{f})-ce^{-2c} \Ia(\widetilde{f})=e^{-2c} \delta \Ia(\widetilde{f}, \widetilde{f})- c \Ia(f).$$ Applying \eqref{ff} to  $\widetilde{f}$ and by  \eqref{homogeneity}, one gets  \begin{align*}
\delta \Ia(\widetilde{f},\widetilde{f}) &=\Big(\frac{n+\alpha}{2}\Big) \Ia(\widetilde{f})-\int_{\mathbb{R}^n}\int_{\mathbb{R}^n}   \frac{\widetilde{\varphi}(x)\widetilde{f}(x)\widetilde{f}(y)}{ |x-y|^{n-\alpha}}dxdy\\ 
&=e^{2c}\Big(\frac{n+\alpha}{2}+c\Big) \Ia(f)-e^{2c} \int_{\mathbb{R}^n}\int_{\mathbb{R}^n}   \frac{\varphi(x)f(x)f(y)}{ |x-y|^{n-\alpha}}dxdy.   
\end{align*} This further implies that  
 \begin{align*}
 \delta \Ia(f,f)&=e^{-2c}\delta \Ia(\widetilde{f},\widetilde{f})-c\Ia(f) \\
 &=\Big(\frac{n+\alpha}{2}+c\Big) \Ia(f)- \int_{\mathbb{R}^n}\int_{\mathbb{R}^n}   \frac{\varphi(x)f(x)f(y)}{|x-y|^{n-\alpha}}dxdy -c\Ia(f)\\
 &=\Big(\frac{n+\alpha}{2} \Big) \Ia(f)- \int_{\mathbb{R}^n}\int_{\mathbb{R}^n}   \frac{\varphi(x)f(x)f(y)}{ |x-y|^{n-\alpha}}dxdy.
\end{align*} Finally,  $\delta \Ia(f,f)\!\in\! (-\infty, \infty)$ follows from  Propositions \ref{finiteness of I_q} and \ref{finite-22-33}.  This completes the proof.
\end{proof}

Let  $\varphi^*$ and $\psi^*$ be such that  $\psi^*=\beta_1 \varphi^*+\beta_2$ for some $\beta_1>0$ and $\beta_2\in \R$.  From \eqref{def-dual}, the lower semi-continuity  and convexity of $\varphi$ and $\psi$, we know  $g=e^{\beta_2} (\beta_1 \bullet f)$, and  from \eqref{def-asp-sum},   \begin{align} (f\oplus t\bullet g)(x)=e^{\beta_2t} e^{-(1+\beta_1 t)\varphi\left(\frac{x}{1+\beta_1 t}\right)}. \label{Asp-sum-spe-case}\end{align}
It follows from \eqref{Asp-sum-spe-case} and  Proposition \ref{properties-a-energy} that  
\begin{align*}
    \Ia(f\oplus t\bullet g)&=e^{2\beta_2t} \int_{\mathbb{R}^n}\int_{\mathbb{R}^n} \frac{ e^{-(1+\beta_1 t)\varphi\left(\frac{x}{1+\beta_1 t}\right)} e^{-(1+\beta_1 t)\varphi\left(\frac{y}{1+\beta_1 t}\right)} }{ |x-y|^{n-\alpha}}dxdy \nonumber \\
&=e^{2\beta_2t} (1+\beta_1t)^{n+\alpha}\int_{\mathbb{R}^n}\int_{\mathbb{R}^n} \frac{ e^{-(1+\beta_1 t)\varphi\left(x\right)} e^{-(1+\beta_1 t)\varphi\left(y\right)} }{ |x-y|^{n-\alpha}}dxdy.
\end{align*}

  With the help of Proposition \ref{prop 4.3}, we obtain  the following  result. 
  \begin{corollary}\label{variational-formula-beta}
      Let $\alpha>0$, $f=e^{-\varphi}\in\LC \ \text{and}\ g=e^{-\psi}\in\LC$. Assume that $\psi^*=\beta_1 \varphi^*+\beta_2$ for some $\beta_1>0$ and $\beta_2 \in\R$.  Then
\begin{eqnarray}
\delta \Ia(f,g)=\beta_1 \delta \Ia(f, f) +\beta_2 \Ia(f). \label{formula 4.6}
\end{eqnarray}
  \end{corollary}

 \subsection{An integral representation of \texorpdfstring{$\delta  \Ia(f, f)$}{}}

In subsection \ref{proportional case}, we gave a   formula for $\delta \Ia(f,f)$, but it is not the right one we want. Based on this result, in this subsection, we will establish the following integral representation for $\delta \Ia(f,f)$.

\begin{proposition}\label{Le.ff-2}
Let $\alpha>0$ and  $f=e^{-\varphi}\in{\rm LC}_n$. If $o\in{\rm int}(K_f)$,  then
    \begin{align}\label{delta.ff}
\delta \Ia(f,f)
=\int_{\R^n}  \varphi^*(\nabla\varphi(x)) I_{\alpha}(f,x) f(x) dx+\int_{\partial K_f} h_{K_f}(\nu_{K_f}(x)) I_{\alpha}(f,x)f(x)d\mathcal{H}^{n-1}(x).
\end{align}
\end{proposition}

To establish Propostion \ref{Le.ff-2}, we need the following lemmas. For convenience, for $\varepsilon\in (0,1)$, let 
\begin{equation}\label{vare-alpha}
  \varepsilon_{\alpha} =\left\{
\begin{array}{lll}
\varepsilon&\text{if} &0<\alpha<n, \\
   0&  \text{if} &\alpha\geq n.
\end{array}
  \right.
\end{equation}

\begin{lemma}\label{Le.div.bound.}
Let $\alpha>0$ and $f=e^{-\varphi}\in{\rm LC}_n$. Then,  for any $0<\varepsilon<1$ and any given $y\in K_f$,
    \begin{align*}     \lim_{R\to+\infty}\int_{B_2^n(y,R)\cap K_f} {\rm div} \left(\frac{xf(x)}{{(|x-y|^2+\varepsilon_{\alpha})^{\frac{n-\alpha}{2}}}}\right) dx=\int_{\partial K_f}\frac{\langle x, \nu_{K_f}(x)\rangle f(x)}{(|x-y|^2+\varepsilon_{\alpha})^{\frac{n-\alpha}{2}}}d\mathcal{H}^{n-1}(x),
\end{align*}
where ${\rm div}(F)$ denotes the divergence of a vector field $F$ and $\varepsilon_{\alpha}$ is defined in \eqref{vare-alpha}.
\end{lemma}
\begin{proof}
The divergence theorem yields that for any given $y\in K_f$,
\begin{align*}
&\lim_{R\to+\infty}\int_{B_2^n(y,R)\cap K_f} {\rm div} \left(\frac{xf(x)}{{(|x-y|^2+\varepsilon_{\alpha})^{\frac{n-\alpha}{2}}}}\right) dx\nonumber\\
 &\quad =\lim_{R\to+\infty}\int_{\partial (B_2^n(y, R)\cap K_f)}\frac{\langle x, \nu_{B_2^n(y, R)\cap K_f}(x)\rangle f(x)}{(|x-y|^2+\varepsilon_{\alpha})^{\frac{n-\alpha}{2}}}d\mathcal{H}^{n-1}(x) \nonumber \\
 &\quad =\lim_{R\to+\infty}\int_{\Xi_1(R)}\frac{\langle x, \nu_{B_2^n(y, R) }(x)\rangle f(x)}{(|x-y|^2+\varepsilon_{\alpha})^{\frac{n-\alpha}{2}}}d\mathcal{H}^{n-1}(x) +\lim_{R\to+\infty}\int_{\Xi_2(R)}\frac{\langle x, \nu_{  K_f}(x)\rangle f(x)}{(|x-y|^2+\varepsilon_{\alpha})^{\frac{n-\alpha}{2}}}d\mathcal{H}^{n-1}(x),
\end{align*}
where 
$$
\Xi_1(R)= \partial \big(B_2^n(y, R)\big) \bigcap \partial \big(B_2^n(y, R)\cap K_f\big)\quad \text{and}\quad \Xi_2(R)=  \partial \big(B_2^n(y, R)\cap K_f\big)\setminus \Xi_1(R).
$$    
If  $\Xi_1(R)\subseteq \partial B_2^n(y,R)$ is not empty, one has
\begin{align*}
0\leq\lim_{R\to+\infty}\int_{\Xi_1(R)}\frac{\langle x, \nu_{B_2^n(y, R) }(x)\rangle f(x)}{(|x-y|^2+\varepsilon_{\alpha})^{\frac{n-\alpha}{2}}}d\mathcal{H}^{n-1}(x)\leq\lim_{R\to+\infty}\int_{S^{n-1}}\frac{(R+\langle y,  v\rangle) f(Rv+y)}{(R^2+\varepsilon_{\alpha})^{\frac{n-\alpha}{2}}}R^{n-1}dv,
\end{align*}
where we have used the substitutuion $x=y+Rv$ for $v\in S^{n-1}$. We can assume that $R>2|y|$ for fixed $y\in\R^n$ which ensures that $0\leq R+\langle v, y\rangle<2R$. Note that $(R^2+\varepsilon_{\alpha})^{\frac{\alpha-n}{2}}\leq   R^{\alpha-n}$ for any $\alpha>0$. Combining with \eqref{In.CF13-11}, one has 
 \begin{align*}   \lim_{R\to+\infty}\int_{S^{n-1}}\frac{(R+\langle y, v\rangle)f(Rv+y)}{(R^2+\varepsilon_{\alpha})^{\frac{n-\alpha}{2}}}R^{n-1}dv  
 &\leq \lim_{R\to+\infty} 2  R^{\alpha} \int_{S^{n-1}} e^{-b|Rv+y|-c}dv\nonumber \\&\leq \lim_{R\to+\infty} 2  R^{\alpha} \int_{S^{n-1}} e^{-\frac{bR}{2}-c}dv\nonumber \\&
=2 n\omega_n\lim_{R\to+\infty} R^{\alpha}   e^{-\frac{bR}{2}-c} =0,  
\end{align*}
 where in the last inequality we have used that $|Rv+y|\geq R-|y|>\frac{R}{2}$. This means 
\begin{align}\label{R-infty-1}
\lim_{R\to+\infty}\int_{\Xi_1(R)}\frac{\langle x, \nu_{B_2^n(y, R) }(x)\rangle f(x)}{(|x-y|^2+\varepsilon_{\alpha})^{\frac{n-\alpha}{2}}}d\mathcal{H}^{n-1}(x)=0.
\end{align}

 On the other hand, if   $\Xi_2(R)\subset \partial K_f$ is not empty, from \eqref{In.CF13-11}  we have 
\begin{align*}
0\leq\!\lim_{R\to+\infty}\int_{\partial K_f\setminus\Xi_2(R)} \!\left|\frac{\langle x, \nu_{K_f}(x)\rangle f(x)}{(|x\!-\!y|^2\!+\!\varepsilon_{\alpha})^{\frac{n-\alpha}{2}}}\right|d\mathcal{H}^{n-1}(x)\leq \!\lim_{R\to+\infty}\int_{\partial K_f\setminus\Xi_2(R)}\frac{|x| e^{-\frac{b|x|+c}{2}}e^{-\frac{\varphi(x)}{2}}}{(|x\!-\!y|^2\!+\!\varepsilon_{\alpha})^{\frac{n-\alpha}{2}}}d\mathcal{H}^{n-1}(x).
\end{align*}
For any $x\in \partial K_f\setminus\Xi_2(R)$, there exist $u\in S^{n-1}$ and $s>R$ such that $x-y=su$.    As  in the proof of \eqref{R-infty-1}, we can assume that $R>2|y|$ for fixed $y\in\R^n$ which ensures that $|x|\geq |x-y|-|y|>\frac{s}{2}$ and $(|x-y|^2+\varepsilon_{\alpha})^{\frac{\alpha-n}{2}}\leq  s^{\alpha-n}$. Combining with the basic inequality $te^{-\beta t}\leq \frac{1}{\beta e}$ (where  $\beta>0$ and $t\geq 0$), we have 
\begin{align*}
  \lim_{R\to+\infty}\int_{\partial K_f\setminus\Xi_2(R)}\frac{|x| e^{-\frac{b|x|+c}{2}}e^{-\frac{\varphi(x)}{2}}}{(|x\!-\!y|^2+\!\varepsilon_{\alpha})^{\frac{n-\alpha}{2}}}d\mathcal{H}^{n-1}(x) \leq \frac{4 }{be^{1+\frac{c}{2}}}\int_{\partial K_ f}\!\!\! e^{-\frac{\varphi(x)}{2}}d\mathcal{H}^{n-1}(x) \left(\lim_{s\to+\infty}\! s^{\alpha-n}   e^{-\frac{bs}{8}}\!\right) \! =0.
\end{align*}
Here, we have used the fact proved in \cite[Proposition 1.6]{Rotem2022} that $\int_{\partial K_ f}f(x)d\mathcal{H}^{n-1}(x)$ is finite for any $f\in\LC$.  Therefore,
\begin{align}\label{R-infty-2}
\lim_{R\to+\infty}\int_{\Xi_2(R)}\frac{\langle x, \nu_{K_f}(x)\rangle f(x)}{(|x-y|^2+\varepsilon_{\alpha})^{\frac{n-\alpha}{2}}}d\mathcal{H}^{n-1}(x)=\int_{\partial K_f}\frac{\langle x, \nu_{K_f}(x)\rangle f(x)}{(|x-y|^2+\varepsilon_{\alpha})^{\frac{n-\alpha}{2}}}d\mathcal{H}^{n-1}(x).
\end{align}
The desired formula comes from the above two equalities \eqref{R-infty-1} and  \eqref{R-infty-2}. 
\end{proof}

From the definition of Legendre transform (see \eqref{def-dual}) and \eqref{basicequ},  the following holds: for almost all $x\in K_f$ where $\varphi$ is differentiable, 
 \begin{align}
     \langle x, \nabla \varphi(x)\rangle=\varphi(x)+\varphi^*(\nabla\varphi(x)) \geq \varphi(x)-\varphi(o)\geq c-\varphi(o)>-\infty,\label{positive-nablavarphi-1}
 \end{align} where $c$ is the constant given by \eqref{In.CF13-11}. 

\begin{lemma}\label{Le.epsilon.bound.} Let $\alpha>0$ and $f=e^{-\varphi}\in{\rm LC}_n$. Then,  
    \begin{align*} 
\lim_{\varepsilon\to0^+}\!\int_{K_f}\!\! \left(\int_{K_f}
\!\frac{\langle x,\nabla\varphi(x)\rangle f(x)f(y)}{(|x-y|^2+\varepsilon_{\alpha})^{\frac{n-\alpha}{2}}}dx+\!\int_{\partial K_f}\!\!\frac{\langle x, \nu_{K_f}(x)\rangle f(x)f(y)}{(|x-y|^2+\varepsilon_{\alpha})^{\frac{n-\alpha}{2}}}d\mathcal{H}^{n-1}(x)\right)dy =\bigg(\!\frac{n+\alpha}{2}\!\bigg) \Ia(f).
\end{align*}
\end{lemma}
\begin{proof}
Let  $0<\varepsilon<1$ be fixed. For any given $y\in K_f$, the monotone convergence theorem and \eqref{positive-nablavarphi-1} yield that
\begin{align*}
 \int_{K_f}
\frac{\langle x,\nabla\varphi(x)\rangle f(x) }{(|x-y|^2+\varepsilon_{\alpha})^{\frac{n-\alpha}{2}}}dx=\lim_{R\to+\infty}\int_{B_2^n(y,R)\cap K_f} \frac{\langle x,\nabla\varphi(x)\rangle f(x) }{(|x-y|^2+\varepsilon_{\alpha})^{\frac{n-\alpha}{2}}}dx.
\end{align*} From the integration by parts,  the fact that  $f(x)=0$ when $x\notin K_f$ and  Lemma \ref{Le.div.bound.}, one has
\begin{align}
 \int_{ K_f}
\frac{\langle x,\nabla\varphi(x)\rangle f(x) }{(|x-y|^2+\varepsilon_{\alpha})^{\frac{n-\alpha}{2}}}dx
&=- \lim_{R\to+\infty}\int_{B_2^n(y,R)\cap K_f}
\left\langle \frac{x}{{(|x\!-\!y|^2+\varepsilon_{\alpha})^{\frac{n-\alpha}{2}}}},\nabla f(x)\right\rangle dx \nonumber \\\nonumber
&=\lim_{R\to+\infty}\int_{B_2^n(y,R)\cap K_f} {\rm div} \left(\frac{x}{{(|x\!-\!y|^2+\varepsilon_{\alpha})^{\frac{n-\alpha}{2}}}}\!\right) f(x)dx \nonumber \\ 
&\quad \ \ -\lim_{R\to+\infty}\int_{B_2^n(y,R)\cap K_f} {\rm div} \left(\frac{xf(x)}{{(|x\!-\!y|^2+\varepsilon_{\alpha})^{\frac{n-\alpha}{2}}}}\!\right) dx \nonumber \\
&=\lim_{R\to+\infty}\int_{B_2^n(y,R)\cap K_f}{\rm div} \left(\frac{x}{{(|x\!-\!y|^2+\varepsilon_{\alpha})^{\frac{n-\alpha}{2}}}}\!\right) f(x)dx \nonumber\\
&\quad\quad   -\int_{\partial K_f}\frac{\langle x, \nu_{K_f}(x)\rangle f(x)}{(|x-y|^2+\varepsilon_{\alpha})^{\frac{n-\alpha}{2}}}d\mathcal{H}^{n-1}(x),\label{two-terms-bound.}
\end{align}
where we have used the following claim:  for fixed $y\in K_f$ and $\varepsilon\in(0,1)$,
\begin{align}\label{bound-div}
    \int_{K_f}\left|{\rm div} \left(\frac{x}{{(|x\!-\!y|^2+\varepsilon_{\alpha})^{\frac{n-\alpha}{2}}}}\!\right) f(x)\right|dx <\infty.
\end{align} Let us show why \eqref{bound-div} holds. Indeed, a direct calculation and the triangle inequality show that, for fixed $y\in K_f$,\begin{align*}
    &\int_{K_f}\left|{\rm div} \left(\frac{x}{{(|x\!-\!y|^2+\varepsilon_{\alpha})^{\frac{n-\alpha}{2}}}}\!\right) f(x)\right|dx\\
     &\quad =\int_{K_f}\left|\frac{n f(x)}{(|x-y|^2+\varepsilon_{\alpha})^{\frac{n-\alpha}{2}}}-(n-\alpha)\frac{\langle x,x-y\rangle f(x)}{(|x-y|^2+\varepsilon_{\alpha})^{\frac{n-\alpha}{2}+1}} \right|dx \\
    &\quad \leq n\int_{K_f}  \frac{f(x)}{{(|x\!-\!y|^2+\varepsilon_{\alpha})^{\frac{n-\alpha}{2}}}} dx+|n-\alpha|\int_{K_f}  \frac{|x|\cdot |x-y| f(x)}{{(|x\!-\!y|^2+\varepsilon_{\alpha})^{\frac{n-\alpha}{2}+1}}} dx\\
    &\quad \leq n\int_{K_f}  \frac{f(x)}{|x\!-\!y|^{n-\alpha}} dx+|n-\alpha|\int_{K_f}  \frac{|x|\cdot |x-y| f(x)}{{(|x\!-\!y|^2+\varepsilon_{\alpha})^{\frac{n-\alpha}{2}+1}}} dx.
\end{align*}
From  the definition of $\varepsilon_{\alpha}$ (see in \eqref{vare-alpha}), we infer that   if $0<\alpha <n$,
\begin{align}\label{f1}
\int_{K_f}\frac{|x|\cdot |x-y|f(x)}{(|x-y|^2+\varepsilon_{\alpha})^{\frac{n-\alpha}{2}+1}}dx    \le {\varepsilon^{\frac{\alpha-n-1}{2}}}
 \left(\int_{K_f}|x|
f(x)dx\right)<\infty,
\end{align}
 where we have used the fact that $\int_{\mathbb{R}^n}|x|^qf(x)dx<\infty$ for any $f\in{\rm LC}_n$ and $q>0$ (see \cite[Propostion 2.1]{HLXZ} and \cite[Lemma 3.1]{FYZZ}). If $\alpha\geq n$,  from \eqref{In.CF13-11},
\begin{align}\label{control-term-22-bound-0.}
 \int_{K_f}\frac{|x|\cdot |x-y|f(x)}{(|x-y|^2+\varepsilon_{\alpha})^{\frac{n-\alpha}{2}+1}}dx &=  \int_{K_f}\frac{|x|  f(x)}{|x-y|^{ n-\alpha+1}}dx\nonumber \\
 &\leq  \frac{1}{b}\int_{K_f
}\frac{ \varphi(x) f(x)}{|x-y|^{n-\alpha+1}}
dx -\frac{c}{b}\int_{K_f
}\frac{   f(x)}{|x-y|^{n-\alpha+1}}
dx.
\end{align}
Our claim \eqref{bound-div} follows  from   Proposition \ref{estimate-riesz potential}, \eqref{f1} and \eqref{control-term-22-bound-0.}. 

Due to \eqref{two-terms-bound.},  \eqref{bound-div} and the monotone convergence theorem, we infer that 
\begin{align}
 &\int_{ K_f}\left(\int_{ K_f}\frac{\langle x,\nabla\varphi(x)\rangle f(x) }{(|x-y|^2+\varepsilon_{\alpha})^{\frac{n-\alpha}{2}}}dx+\int_{\partial K_f}\frac{\langle x, \nu_{K_f}(x)\rangle f(x)}{(|x-y|^2+\varepsilon_{\alpha})^{\frac{n-\alpha}{2}}}d\mathcal{H}^{n-1}(x)\right)f(y)dy\nonumber \\
&\quad = n\int_{  K_f}\int_{  K_f}\frac{ f(x)f(y)}{(|x-y|^2+\varepsilon_{\alpha})^{\frac{n-\alpha}{2}}}dxdy-(n-\alpha)\int_{  K_f}\int_{  K_f}\frac{\langle x,x-y\rangle f(x)f(y)}{(|x-y|^2+\varepsilon_{\alpha})^{\frac{n-\alpha}{2}+1}} dxdy.\label{div-calculation-1-bound.}
\end{align}
By switching the roles of $x$ and $y$,
\begin{align}
 &\int_{ K_f}\left(\int_{ K_f}\frac{\langle y,\nabla\varphi(y)\rangle f(y) }{(|x-y|^2+\varepsilon_{\alpha})^{\frac{n-\alpha}{2}}}dy+\int_{\partial K_f}\frac{\langle y, \nu_{K_f}(y)\rangle f(y)}{(|x-y|^2+\varepsilon_{\alpha})^{\frac{n-\alpha}{2}}}d\mathcal{H}^{n-1}(y)\right)f(x)dx\nonumber \\
&\quad = n\int_{  K_f}\int_{  K_f}\frac{ f(x)f(y)}{(|x-y|^2+\varepsilon_{\alpha})^{\frac{n-\alpha}{2}}}dydx-(n-\alpha)\int_{  K_f}\int_{  K_f}\frac{\langle y,y-x\rangle f(x)f(y)}{(|x-y|^2+\varepsilon_{\alpha})^{\frac{n-\alpha}{2}+1}} dydx.\label{div-calculation-2-bound.}
\end{align}
Indeed, \eqref{f1} and \eqref{control-term-22-bound-0.} ensure that Fubini's Theorem can be used to obtain that, for any $\alpha >0$ and for any fixed $0<\varepsilon\le 1$,  
\begin{align}\label{add-4.29-bound.}
&\int_{K_f}\int_{K_f}\frac{\langle x,x-y\rangle f(x)f(y) }{(|x-y|^2+\varepsilon_{\alpha})^{\frac{n-\alpha}{2}+1}}dxdy+
\int_{K_f}\int_{K_f}\frac{\langle -y,x-y\rangle f(x)f(y)}{(|x-y|^2+\varepsilon_{\alpha})^{\frac{n-\alpha}{2}+1}}dydx \nonumber\\
&\quad \ \ =\int_{K_f}\int_{K_f}\frac{\langle x,x-y\rangle f(x)f(y)}{(|x-y|^2+\varepsilon_{\alpha})^{\frac{n-\alpha}{2}+1}}dxdy+
\int_{K_f}\int_{K_f}\frac{\langle -y,x-y\rangle f(x)f(y)}{(|x-y|^2+\varepsilon_{\alpha})^{\frac{n-\alpha}{2}+1}}dxdy \nonumber\\
&\quad\ \  =\int_{K_f}\int_{K_f}\frac{| x-y|^2f(x)f(y)}{(|x-y|^2+\varepsilon_{\alpha})^{\frac{n-\alpha}{2}+1}}dxdy<\infty.
\end{align}
Combining   \eqref{div-calculation-1-bound.}, \eqref{div-calculation-2-bound.}, \eqref{add-4.29-bound.} with  the dominated convergence theorem, we obtain
\begin{align*}\nonumber
&\lim_{\varepsilon\to0^+}\int_{K_f}\left(\int_{K_f}
\frac{\langle x,\nabla\varphi(x)\rangle f(x)f(y)}{(|x-y|^2+\varepsilon_{\alpha})^{\frac{n-\alpha}{2}}}dx+\int_{\partial K_f}\frac{\langle x, \nu_{K_f}(x)\rangle f(x)f(y)}{(|x-y|^2+\varepsilon_{\alpha})^{\frac{n-\alpha}{2}}}d\mathcal{H}^{n-1}(x)\right)dy\nonumber\\
&\quad=\lim_{\varepsilon\to0^+}\frac{1}{2}\left[\int_{K_f}\int_{K_f}
\frac{\langle x,\nabla\varphi(x)\rangle f(x)f(y)}{(|x-y|^2+\varepsilon_{\alpha})^{\frac{n-\alpha}{2}}}dxdy+\int_{K_f}\int_{\partial K_f}\frac{\langle x, \nu_{K_f}(x)\rangle f(x)f(y)}{(|x-y|^2+\varepsilon_{\alpha})^{\frac{n-\alpha}{2}}}d\mathcal{H}^{n-1}(x)dy\right.\\
&\quad\quad\quad    \quad +\left.\int_{K_f}\int_{K_f}
\frac{\langle y,\nabla\varphi(y)\rangle f(x)f(y)}{(|x-y|^2+\varepsilon_{\alpha})^{\frac{n-\alpha}{2}}}dydx+\int_{K_f}\int_{\partial K_f}\frac{\langle y, \nu_{K_f}(y)\rangle f(x)f(y)}{(|x-y|^2+\varepsilon_{\alpha})^{\frac{n-\alpha}{2}}}d\mathcal{H}^{n-1}(y)dx\right]\\
&\quad=\lim_{\varepsilon\to0^+}\int_{K_f}\int_{K_f}\left(\frac{n}{(|x-y|^2+\varepsilon_{\alpha})^{\frac{n-\alpha}{2}}}-\Big(\frac{n-\alpha}{2}\Big)\frac{|x-y|^2}{(|x-y|^2+\varepsilon_{\alpha})^{\frac{n-\alpha}{2}+1}}\right)
f(x)f(y)dxdy\nonumber\\
&\quad =n\lim_{\varepsilon\to0^+}\int_{K_f}\int_{K_f}\frac{f(x)f(y)}{(|x-y|^2+\varepsilon_{\alpha})^{\frac{n-\alpha}{2}}}dxdy-
\Big(\frac{n-\alpha}{2}\Big)
\lim_{\varepsilon\to0^+}\int_{K_f}\int_{K_f}\frac{|x-y|^2f(x)f(y)}{(|x-y|^2+\varepsilon_{\alpha})^{\frac{n-\alpha}{2}+1}}dxdy
\\\nonumber
&\quad=\left(\frac{n+\alpha}{2}\right) \Ia(f).
\end{align*}
This is the desired result.
\end{proof}

 Now, we can complete  the proof of Proposition \ref{Le.ff-2}.

\begin{proof}[Proof of Proposition \ref{Le.ff-2}] If  $\varphi(o)=0$, 
 by the fact that  $o\in{\rm int}(K_f)$, the monotone convergence theorem yields, for any $\alpha>0$,   that 
\begin{align}\label{eq.limt.4}
    \lim_{\varepsilon\to 0^+}\!\int_{K_f}\int_{\partial K_f}\!\!\frac{\langle x, \nu_{K_f}(x)\rangle f(x)f(y)}{(|x-y|^2+\varepsilon_{\alpha})^{\frac{n-\alpha}{2}}}d\mathcal{H}^{n-1}(x)dy &=\!\int_{K_f}\!\int_{\partial K_f}\!\!\!\frac{\langle x, \nu_{K_f}(x)\rangle f(x)f(y)}{|x-y|^{n-\alpha}}d\mathcal{H}^{n-1}(x)dy,\\ \label{eq1}
\lim_{\varepsilon\to0^+}\!\int_{K_f}\int_{K_f}\frac{\varphi^\ast(\nabla\varphi(x))f(x)f(y)}{(|x-y|^2+\varepsilon_{\alpha})^{\frac{n-\alpha}{2}}}dxdy
&=\!\int_{K_f}\int_{K_f}\frac{\varphi^\ast(\nabla\varphi(x))f(x)f(y)}{|x-y|^{n-\alpha}}dxdy . 
\end{align}
 Combining \eqref{eq.limt.4},  \eqref{eq1},    Lemma \ref{Le.epsilon.bound.} and Proposition \ref{prop 4.3}, we  infer that
\begin{align*} & \int_{K_f}\int_{K_f}\frac{\varphi^\ast(\nabla\varphi(x))f(x)f(y)}{|x-y|^{n-\alpha}}dxdy+\int_{K_f}\int_{\partial K_f}\frac{\langle x, \nu_{K_f}(x)\rangle f(x)f(y)}{|x-y|^{n-\alpha}}d\mathcal{H}^{n-1}(x)dy\\
&\quad =\lim_{\varepsilon\to 0^{+}}\left(\int_{K_f}\int_{K_f}\frac{\varphi^\ast(\nabla\varphi(x))f(x)f(y)}{(|x-y|^2+\varepsilon_{\alpha})^{\frac{n-\alpha}{2}}}dxdy+\int_{K_f}\int_{\partial K_f}\frac{\langle x, \nu_{K_f}(x)\rangle f(x)f(y)}{(|x-y|^2+\varepsilon_{\alpha})^{\frac{n-\alpha}{2}}}d\mathcal{H}^{n-1}(x)dy\right)\\
& \quad =\lim_{\varepsilon\to 0^{+}}\left(\int_{K_f}\int_{K_f}
\frac{\langle x,\nabla\varphi(x)\rangle f(x)f(y)}{(|x-y|^2+\varepsilon_{\alpha})^{\frac{n-\alpha}{2}}}dxdy+\int_{K_f}\int_{\partial K_f}\frac{\langle x, \nu_{K_f}(x)\rangle f(x)f(y)}{(|x-y|^2+\varepsilon_{\alpha})^{\frac{n-\alpha}{2}}}d\mathcal{H}^{n-1}(x)dy\right)\\
&\quad\quad-\lim_{\varepsilon\to 0^{+}}\int_{K_f}\int_{K_f}   \frac{\varphi(x)f(x)f(y)}{(|x-y|^2+\varepsilon_{\alpha})^{\frac{n-\alpha}{2}}}dxdy\\
&\quad =\bigg(\frac{n+\alpha}{2}\bigg)   \Ia(f)-\int_{K_f}\int_{K_f}   \frac{\varphi(x)f(x)f(y)}{ |x-y|^{n-\alpha}}dxdy =\delta \Ia(f,f),
\end{align*}
where we have used the fact that 
 \begin{align*} 
 \lim_{\varepsilon\to 0^{+}}\int_{K_f}\int_{K_f}   \frac{\varphi(x)f(x)f(y)}{(|x-y|^2+\varepsilon_{\alpha})^{\frac{n-\alpha}{2}}}dxdy=\int_{K_f}\int_{K_f}   \frac{\varphi(x)f(x)f(y)}{ |x-y|^{n-\alpha}}dxdy<+\infty.
 \end{align*}
This is a directly result from the  monotone convergence theorem, \eqref{In.CF13-11} and Proposition \ref{finite-22-33}.
 
  If   $\varphi(o)\neq 0$, we set $\widetilde{\varphi}=\varphi-\varphi(o)$  and $\widetilde{f}=e^{-\widetilde{\varphi}}$.  Then, \eqref{delta.ff} holds for $\widetilde{f}$.  A basic calculation and \eqref{add-1} ensure \eqref{delta.ff} still  hold for $f=e^{-\varphi}$.
 \end{proof}

\section{The variational formulas for the Riesz \texorpdfstring{$\alpha$}{}-energy  and the integral expression of \texorpdfstring{$\delta\Ia(f,g)$}. }\label{Section:-5}
In this section, we will calculate $\delta\Ia(f,g)$ for   $f=e^{-\varphi}\in \LC$ and  $g=e^{-\psi} \in \LC$ under condition  \eqref{nat-condition-1}. 
We first show  that Corollary \ref{var.bound.}  follows from    Theorem \ref{var.bound.-1} under the  condition ${\rm int}(K_f)\neq \emptyset$.

\begin{proof}[Proof of Corollary \ref{var.bound.}]
Let $f=e^{-\varphi}\in\LC$ and ${\rm int}(K_f)\neq \emptyset$. Assume that $x_0\in {\rm int}(K_f)$, and let $\bar{f}(x)=f(x+x_0)$ for $x\in\R^n$.  Clearly,  $K_{\bar{f}}=K_{f}-x_0$ and then $o\in {\rm int}(K_{\bar{f}})$.  Let  $\overline{\varphi}=-\log \bar{f}$, and from \eqref{def-dual}, $\overline{\varphi}^*(y) =\varphi^*(y)-\langle x_0,y\rangle$ for any $y\in\R^n$. Then condition \eqref{nat-condition-1} for $\varphi$ and $\psi$ ensure that $\overline{\varphi}$ and $\psi$ satisfy \eqref{nat-condition-1} (but with different constants) as well. To this end, from \eqref{nat-condition-1} for $\varphi$ and $\psi$, one has,  for all $y\in \Rn$, \begin{align}\label{formula-add-1}
\psi^*(y)\leq \beta_1\varphi^*(y)+\beta_2=\beta_1 (\overline{\varphi}^*(y)+\langle x_0, y\rangle)+\beta_2 \leq \beta_1 (\overline{\varphi}^*(y)+ |x_0|\cdot |y|)+\beta_2 
\end{align} for some constants $\beta_1>0$ and $\beta\in \R.$ Note that $o\in {\rm int}(K_{\bar{f}})$, then there exist constants $\overline{b}>0$ and $\overline{c}\in \R$ such that $\overline{\varphi}^*(y)\geq \overline{b}|y|+\overline{c}. $ Together with \eqref{formula-add-1}, one gets, for all $y\in \Rn,$
 \begin{align*}
\psi^*(y)\leq   \beta_1 (\overline{\varphi}^*(y)+ |x_0|\cdot |y|)+\beta_2  \leq  \frac{\beta_1(\overline{b}+|x_0|)}{\overline{b}} \overline{\varphi}^*(y)+\frac{\beta_2\overline{b}-\overline{c}|x_0|\beta_1}{\overline{b}}. 
\end{align*} That is, $\overline{\varphi}$ and $\psi$ satisfy \eqref{nat-condition-1} (with different constants). It then follows from Theorem \ref{var.bound.-1} that 
\begin{align}
\delta\Ia(\bar{f},g)=\!\int_{K_{\bar{f}}} \!\int_{K_{\bar{f}}}\!\!\frac{\psi^*(\nabla\overline{\varphi}(x))  \bar{f}(y)\bar{f}(x)}{|x-y|^{n-\alpha}}dxdy
+\int_{\partial K_{\bar{f}}}\!\int_{K_{\bar{f}}}\!\!\frac{ h_{K_g}(\nu_{K_{\bar{f}}}(x)) \bar{f}(y)\bar{f}(x)}{|x-y|^{n-\alpha}}dyd\mathcal{H}^{n-1}(x), \label{new-5.2}
\end{align}
where $g=e^{-\psi}\in \LC$. From \eqref{def-asp-sum-1} and \eqref{def-dual}, one has $(\bar{f}\oplus t\bullet g) (x)=(f\oplus t\bullet g) (x+x_0)$ for any $x\in\R^n$. Then,  \eqref{first variation} and \eqref{translation invariance}  deduce that 
\begin{align}\label{cor1}
\delta\Ia(\bar{f},g)=\frac{1}{2} \lim_{t\to0^{+}}\frac{\Ia(\bar{f}\oplus t\bullet g)-\Ia(\bar{f})}{t}=\frac{1}{2}\lim_{t\to0^{+}}\frac{\Ia(f\oplus t\bullet g)-\Ia(f)}{t}=\delta\Ia(f,g).
\end{align}
Using  the fact that $K_{\bar{f}}=K_{f}-x_0$, one has 
\begin{align}\label{cor2}
\int_{K_{\bar{f}}}\int_{K_{\bar{f}}}\frac{\psi^*(\nabla\bar{\varphi}(x))  \bar{f}(y)\bar{f}(x)}{|x-y|^{n-\alpha}}dxdy&=\int_{ K_{f}-x_0}\int_{ K_{f}-x_0}\frac{\psi^*(\nabla\varphi(x+x_0))   f(y+x_0)f(x+x_0)}{|x-y|^{n-\alpha}}dxdy\nonumber \\
&=\int_{ K_{f}}\int_{ K_{f}}\frac{\psi^*(\nabla\varphi(x))   f(y)f(x)}{|x-y|^{n-\alpha}}dxdy.
\end{align}
Similarly, from the fact that $\nu_{K_{\bar{f}}}(x)=\nu_{K_{f}}(x+x_0)$,   one has
\begin{align}\label{cor3}
\int_{\partial K_{\bar{f}}}\int_{K_{\bar{f}}}\frac{ h_{K_g}(\nu_{K_{\bar{f}}}(x)) \bar{f}(y)\bar{f}(x) }{|x-y|^{n-\alpha}}dyd\mathcal{H}^{n-1}(x) =\int_{\partial K_{f}}\int_{K_{f}}\frac{ h_{K_g}(\nu_{K_{f}}(x)) f(y)f(x) }{|x-y|^{n-\alpha}}dyd\mathcal{H}^{n-1}(x).
\end{align}
The desired result follows from \eqref{new-5.2}, \eqref{cor1}, \eqref{cor2}  and \eqref{cor3}.
\end{proof}

Before we  prove Theorem \ref{var.bound.-1}, we need some lemmas involving the function 
\begin{align}\label{hat-f}
\widehat{f}_t(x)=f\oplus (t\bullet e^{-(\beta_1 \varphi^*+\beta_2)^*})(x)=e^{\beta_2 t-(1+\beta_1t)\varphi(\frac{x}{1+\beta_1t})}
\end{align}
for some $\beta_1>0$ and $\beta_2\in\R$.   
For any $r(t)\in[1,1+\beta_1t)$, from the convexity of $\varphi$, we get $$(1+\beta_1 t)\varphi\left(\frac{r(t)x}{1+\beta_1t}\right)\leq r(t)\varphi(x)+(1+\beta_1t-r(t))\varphi(o).$$ Assume that $\lim_{t\to0^{+}}r(t)=1^+$. 
Since $o\in  \mathrm{int}(K_f)$ (which implies $\varphi(o)<\infty$), one gets, 
$$\limsup_{t\to0^{+}}(1+\beta_1 t)\varphi\left(\frac{r(t)x}{1+\beta_1t}\right)\leq  \varphi(x) \ \ \  \mathrm{for\ any} \ x\in K_f.$$ 
One the other hand,  the lower semi-continuity  of $\varphi$ implies that  $$\liminf_{t\to0^{+}} (1+\beta_1 t)\varphi\left(\frac{r(t)x}{1+\beta_1t}\right) \geq \varphi(x)  \ \ \  \mathrm{for\ any} \ x\in K_f.$$
Therefore, for any $r(t)\in[1,1+\beta_1t)$ such that $\lim_{t\to0^{+}}r(t)=1^+,$ and for any $x\in K_f$, one has 
\begin{align}\label{hat-f-to0}
\lim_{t\to0^{+}} \widehat{f}_t(r(t)x)=\lim_{t\to0^{+}} e^{\beta_2 t-(1+\beta_1t)\varphi(\frac{r(t)x}{1+\beta_1t})}=  f(x).
\end{align}
As $o\in{\rm int}(K_f)$,  one can define $\rho_{K_f}$, the radial function of $K_f$ (not necessary bounded), by $\rho_{K_f}(u)=\sup\{t>0:tu\in K_f\}$. Note that $\rho_{K_f}$ may be infinite and let  $$\Omega_f=\{u\in S^{n-1}:\rho_{K_f}(u)<\infty\}.$$
Write $\Omega_f^c=\sphere\setminus\Omega_f$.

\begin{lemma}\label{Le.J1}
Let $f=e^{-\varphi}\in \LC$ and   $\widehat{f}_t $   be given in \eqref{hat-f}.  If $o\in{\rm int}(K_f)$, then for any $\alpha>0$
\begin{align}\label{J1}
\lim_{t\to 0^+}\int_{K_{\widehat{f}_t}\setminus K_f}\int_{K_{\widehat{f}_t}\setminus K_f}\frac{\widehat{f}_t(x)\widehat{f}_t(y)}{t|x-y|^{n-\alpha}}dxdy=0.
\end{align}
\end{lemma}
\begin{proof}
 From \cite[Proposition 2.1 (iii)] {CF13},  it holds    $K_{\widehat{f}_t}= (1+t\beta_1)K_f$.  When $K_f=\Rn$, \eqref{J1} is trivial.  If  $\alpha\geq n$, from the triangle inequality, one has 
     $|x-y|^{\alpha-n}\leq 2^{\alpha-n}(|x|^{\alpha-n}+|y|^{\alpha-n})$
 and 
\begin{align}\label{J1-0}
 \int_{K_{\widehat{f}_t}\setminus K_f}\int_{K_{\widehat{f}_t}\setminus K_f}\frac{\widehat{f}_t(x)\widehat{f}_t(y)}{|x-y|^{n-\alpha}}dxdy 
&  \leq \int_{K_{\widehat{f}_t}\setminus K_f}\int_{K_{\widehat{f}_t}\setminus K_f}\widehat{f}_t(x)\widehat{f}_t(y)2^{\alpha-n}(|x|^{\alpha-n}+|y|^{\alpha-n})dxdy\nonumber\\
& =2^{\alpha-n+1}\left(\int_{K_{\widehat{f}_t}\setminus K_f} \widehat{f}_t(y)dy\right)\cdot \left(\int_{K_{\widehat{f}_t}\setminus K_f} \widehat{f}_t(x)|x|^{\alpha-n}dx\right).
\end{align}
From the fact that $o\in {\rm int}(K_f)$ and \eqref{In.CF13-11}, we know that there exist $c\in\R$ and $b>0$ such that
\begin{align} \label{J1-1}
\int_{K_{\widehat{f}_t}\setminus K_f} \widehat{f}_t(x)|x|^{\alpha-n}dx&\leq e^{t\beta_2-(1+t\beta_1)c}\int_{K_{\widehat{f}_t}\setminus K_f}  e^{-b|x|}|x|^{\alpha-n}dx \nonumber\\
&= e^{t\beta_2-(1+t\beta_1)c}\int_{ \Omega_f}\int_{\rho_{K_f}(u)}^{(1+t\beta_1)\rho_{K_f}(u)}e^{-br}r^{\alpha-1}drdu\nonumber\\
&\leq e^{t\beta_2-(1+t\beta_1)c}\int_{\Omega_f} e^{-b\rho_{K_f}(u)} \int_{\rho_{K_f}(u)}^{(1+t\beta_1)\rho_{K_f}(u)} r^{\alpha-1}drdu\nonumber\\
&=\alpha^{-1}e^{t\beta_2-(1+t\beta_1)c}\left((1+t\beta_1)^{\alpha}-1\right)\int_{\Omega_f}e^{-b\rho_{K_f}(u)}\rho_{K_f}(u)^{\alpha}du.
\end{align}
A same argument helps us to obtain that
\begin{align}\label{J1-2}
\int_{K_{\widehat{f}_t}\setminus K_f} \widehat{f}_t(y)dy&\leq n^{-1}e^{t\beta_2-(1+t\beta_1)c}\left((1+t\beta_1)^{n}-1\right)\int_{\Omega_f}e^{-b\rho_{K_f}(u)}\rho_{K_f}(u)^{n}du.
\end{align}
Therefore, the desired result  follows from \eqref{J1-0}, \eqref{J1-1}, \eqref{J1-2} and  the  fact that: for any $\alpha>0$, \begin{align} \label{M}
M_\alpha:=\int_{\Omega_f}e^{-b\rho_{K_f}(u)}\rho_{K_f}(u)^{\alpha}du 
&=\alpha \int_{\Omega_f}e^{-b\rho_{K_f}(u)} \int_{0}^{\rho_{K_f}(u)} r^{\alpha-1}drdu \nonumber\\
&\leq \alpha \int_{\Omega_f} \int_{0}^{\rho_{K_f}(u)} e^{-br}r^{\alpha-1}drdu \nonumber\\
&\leq \alpha\int_{S^{n-1}} \int_{0}^{\infty } e^{-br}r^{\alpha-1}drdu  \nonumber \\
&= \alpha n\omega_n b^{-\alpha}\Gamma (\alpha)<\infty.
\end{align}

Next,   consider the case of $0<\alpha <n$.  For  any fixed $\varepsilon>0$,   
it holds
\begin{align}\label{leqn1}
0&\leq \lim_{t\to 0^+}\int_{K_{\widehat{f}_t}\setminus K_f}\int_{K_{\widehat{f}_t}\setminus K_f}\frac{\widehat{f}_t(x)\widehat{f}_t(y)}{t|x-y|^{ n-\alpha}}dxdy\nonumber \\
&= \lim_{t\to 0^+}\int_{K_{\widehat{f}_t}\setminus K_f}\int_{B_2^n(y,\varepsilon)\cap(K_{\widehat{f}_t}\setminus K_f)}\frac{\widehat{f}_t(x)\widehat{f}_t(y)}{t|x-y|^{ n-\alpha}}dxdy \nonumber \\
    &\quad +\lim_{t\to 0^+}\int_{K_{\widehat{f}_t}\setminus K_f}\int_{B_2^n(y,\varepsilon)^c\cap(K_{\widehat{f}_t}\setminus K_f)}\frac{\widehat{f}_t(x)\widehat{f}_t(y)}{t|x-y|^{ n-\alpha}}dxdy.
        \end{align}
  By \eqref{upperboundoff0}, \eqref{J1-2} and \eqref{M}, we have
     \begin{align}\label{leqn2}
 \lim_{t\to 0^+}\int_{K_{\widehat{f}_t}\setminus K_f}\int_{B_2^n(y,\varepsilon)^c\cap(K_{\widehat{f}_t}\setminus K_f)}\frac{\widehat{f}_t(x)\widehat{f}_t(y)}{t|x-y|^{ n-\alpha}}dxdy
 &  \leq \lim_{t\to 0^+}\int_{K_{\widehat{f}_t}\setminus K_f}\int_{B_2^n(y,\varepsilon)^c\cap(K_{\widehat{f}_t}\setminus K_f)} \frac{\widehat{f}_t(x)\widehat{f}_t(y)}{t  \varepsilon^{ n-\alpha}}dxdy \nonumber\\
&  \leq  \varepsilon^{ \alpha-n}\lim_{t\to 0^+}\int_{K_{\widehat{f}_t}\setminus K_f}\int_{ K_{\widehat{f}_t}\setminus K_f} \frac{\widehat{f}_t(x)\widehat{f}_t(y)}{t}dxdy \nonumber\\
& \leq \varepsilon^{ \alpha-n}M_n^2\lim_{t\to 0^{+}}\frac{n^{-2}e^{2t\beta_2-2(1+t\beta_1)c}\left((1+t\beta_1)^{n}-1\right)^2}{t} \nonumber\\
& =0.
\end{align}      
From  \eqref{In.CF13-11}, \eqref{J1-2} and \eqref{M}, one has
   \begin{align}\nonumber
 &\lim_{t\to 0^+}\int_{K_{\widehat{f}_t}\setminus K_f}\int_{B_2^n(y,\varepsilon)\cap(K_{\widehat{f}_t}\setminus K_f)}\frac{\widehat{f}_t(x)\widehat{f}_t(y)}{t|x-y|^{ n-\alpha}}dxdy   \\\nonumber
 &\quad \leq\lim_{t\to 0^+}\frac{e^{t\beta_2-(1+t\beta_1)c}}{t}  \int_{K_{\widehat{f}_t}\setminus K_f} \widehat{f}_t(y) \left(\int_{B_2^n(y,\varepsilon)\cap(K_{\widehat{f}_t}\setminus K_f)} |x-y|^{ \alpha-n}dx\right)dy \\\nonumber
  &\quad \leq \int_{ B_2^n(o,\varepsilon)} |z|^{ \alpha-n}dz  \cdot  \lim_{t\to 0^+} \left(\frac{e^{t\beta_2-(1+t\beta_1)c} }{t}  \int_{K_{\widehat{f}_t}\setminus K_f} \widehat{f}_t(y) dy\right) \\ 
  &\quad \leq  \alpha \omega_n e^{-2c} M_n\varepsilon^{\alpha} \lim_{t\to 0^+}\frac{e^{2t\beta_2 +2t\beta_1c} \left((1+t\beta_1)^n-1\right) }{t}\nonumber\\
  &\quad = (n\beta_1\alpha \omega_n e^{-2c} M_n)\varepsilon^{\alpha}.\label{leqn3}
        \end{align}
We infer from \eqref{leqn1}, \eqref{leqn2} and \eqref{leqn3} that, for any $\varepsilon>0$, \begin{align*}
0&\leq \lim_{t\to 0^+}\int_{K_{\widehat{f}_t}\setminus K_f}\int_{K_{\widehat{f}_t}\setminus K_f}\frac{\widehat{f}_t(x)\widehat{f}_t(y)}{t|x-y|^{ n-\alpha}}dxdy \leq (n\beta_1\alpha \omega_n e^{-2c} M_n)\varepsilon^{\alpha}.
        \end{align*}
Hence, the arbitrariness of $\varepsilon$ yields  the desired result for $0<\alpha<n$.
\end{proof}

According to \eqref{In.CF13-11}, if $f=e^{-\varphi}\in{\rm LC}_n$, then $\varphi^*(o)= -\inf \varphi \leq -c<\infty.$ If $\psi$ satisfies condition  \eqref{nat-condition-1}, then $\psi^*(o)< \infty.$ Applying  Lemma \ref{rotem-111} to  $f=e^{-\varphi}\in \LC$ and $g=e^{-\psi}\in\LC$, one gets 
\begin{align}\label{rotem-formula1}
    \frac{d}{dt}(\varphi^\ast+t\psi^\ast)^\ast(x)\bigg|_{t=0^+}
    =-\psi^\ast(\nabla\varphi(x)),
\end{align}
for any $x\in\mathbb{R}^n$ where $\varphi $
is differentiable.

\begin{lemma}\label{Le.Iff-bound.}
  Let $f=e^{-\varphi}\in \LC$ with $o\in{\rm int}(K_f)$ and $\widehat{f}_t $    be given in \eqref{hat-f}.   For any $\alpha>0$, $u\in\sphere, y\in K_f$, and $t\in(0,1)$,  define
\begin{align}\label{E_t}
E_{t}(u,y)=\!\int_{\rho_{K_f}(u)}^{\rho_{K_{\widehat{f}_t}}(u)}\! \frac{\widehat{f}_t(ru)\widehat{f}_t(y)}{t|r u-y|^{ n-\alpha}}r^{n-1}dr+\!\int_0^{\rho_{K_f}(u)}\frac{(\widehat{f}_t(r u)+f(r u))(\widehat{f}_t(y)\!-\!  f(y))}{2t|r u-y| ^{n-\alpha }} r^{n-1}dr,
\end{align}
where, if $u\in\Omega_f^c$, the first term in the right hand side of \eqref{E_t} is treated as $0$. 
 Then, 
 \begin{align*}
& \lim_{t\to 0^+} \int_{K_{f}}\int_{ \sphere} E_{t }(u,y)  dudy
=\int_{K_{f}}\int_{ \sphere} \lim_{t\to0^{+}} E_{t }(u,y) dudy.
\end{align*}
\end{lemma}
\begin{proof}
Let  $\mathfrak{g}=e^{-(\beta_1 \varphi^*+\beta_2)^*}$.    From Corollary  \ref{variational-formula-beta} and  formula   \eqref{first variation},  one has
\begin{align}\label{int-out0}
&\beta_1 \delta \Ia(f,f)+\beta_2 \Ia(f)=  \delta \Ia(f,  \mathfrak{g}) \nonumber\\
&\quad=\frac{1}{2}\lim_{t\to0^{+}} \frac{1}{t}\left(\int_{K_{\widehat{f}_t}}\int_{K_{\widehat{f}_t}}\frac{\widehat{f}_t(x)\widehat{f}_t(y)}{|x-y|^{n-\alpha}}dxdy-\int_{K_{f}}\int_{K_{f}}\frac{f(x)f(y)}{|x-y|^{n-\alpha}}dxdy \right)\nonumber\\
&\quad= \frac{1}{2}\lim_{t\to 0^+} \int_{K_{\widehat{f}_t}\setminus K_f}\int_{K_{\widehat{f}_t}\setminus K_f}\frac{\widehat{f}_t(x)\widehat{f}_t(y)}{t|x-y|^{ n-\alpha }}dxdy  + \frac{1}{2}\lim_{t\to 0^+}\int_{K_{f}}\int_{K_{\widehat{f}_t}\setminus K_f}\frac{\widehat{f}_t(x)\widehat{f}_t(y)}{t|x-y|^{n-\alpha }}dxdy\nonumber\\&\quad\quad + \frac{1}{2}\lim_{t\to 0^+}\int_{K_{\widehat{f}_t}\setminus K_f}\int_{K_{f}}\frac{\widehat{f}_t(x)\widehat{f}_t(y)}{t|x-y|^{n-\alpha}}dxdy+ \frac{1}{2}\lim_{t\to 0^+} \int_{K_{f}}\int_{ K_f}\frac{\widehat{f}_t(x)\widehat{f}_t(y)-f(x)f(y)}{t|x-y| ^{ n-\alpha }}dxdy. 
\end{align}
For any $t>0$, by the fact that $\widehat{f}_t\in{\rm LC}_n$ and Proposition \ref{finiteness of I_q}, the Fubini theorem  and the role switching of $x$ and $y$ imply  that  
\begin{align}\label{TF10}\int_{K_{f}}\int_{K_{\widehat{f}_t}\setminus K_f}\frac{\widehat{f}_t(x)\widehat{f}_t(y)}{t|x-y|^{n-\alpha }}dxdy = \int_{K_{\widehat{f}_t}\setminus K_f}\int_{K_{f}}\frac{\widehat{f}_t(x)\widehat{f}_t(y)}{t|x-y|^{n-\alpha}}dxdy.
\end{align}
From   \eqref{int-out0}, \eqref{TF10} and Lemma \ref{Le.J1},  one has 
\begin{align}\label{int-out0-1}
  \delta \Ia(f, \mathfrak{g})
 &\!=\! \lim_{t\to 0^+}\!\int_{K_{f}}\int_{K_{\widehat{f}_t}\setminus K_f}\frac{\widehat{f}_t(x)\widehat{f}_t(y)}{t|x-y|^{n-\alpha }}dxdy+ \frac{1}{2}\lim_{t\to 0^+} \!\int_{K_{f}}\int_{ K_f}\frac{\widehat{f}_t(x)\widehat{f}_t(y)-f(x)f(y)}{t|x-y| ^{ n-\alpha }}dxdy\nonumber\\
 &\!=\!  \lim_{t\to 0^+}\!\int_{{\rm int}(K_{f})}\!\int_{K_{\widehat{f}_t}\setminus K_f}\!\frac{\widehat{f}_t(x)\widehat{f}_t(y)}{t|x-y|^{n-\alpha }}dxdy\!+\! \frac{1}{2}\lim_{t\to 0^+}\! \int_{K_{f}}\!\int_{ K_f}\!\!\!\!\frac{\widehat{f}_t(x)\widehat{f}_t(y)\!-\!f(x)f(y)}{t|x-y| ^{ n-\alpha }}dxdy. 
\end{align}
Let us consider the first term  in  \eqref{int-out0-1}. Note that $K_{\widehat{f}_t}\setminus K_f$ along the direction $u\in \Omega_{f}^c$ is simply empty. Together with the polar coordinates, one gets, for any fixed $t>0$, 
\begin{align*}
\int_{{\rm int}( K_{f})}\int_{K_{\widehat{f}_t}\setminus K_f}\frac{\widehat{f}_t(x)\widehat{f}_t(y)}{t|x-y|^{n-\alpha }}dxdy   =\int_{{\rm int}( K_{f})}\int_{\Omega_f}\int_{\rho_{K_f}(u)}^{\rho_{K_{\widehat{f}_t}}(u)}\frac{\widehat{f}_t(ru)\widehat{f}_t(y)}{t|ru-y|^{n-\alpha }}r^{n-1}drdudy.
\end{align*}
For convenience, we let 
\begin{align}\label{int}
\int_{K_{f}}\int_{K_{\widehat{f}_t}\setminus K_f}\frac{\widehat{f}_t(x)\widehat{f}_t(y)}{t|x-y|^{n-\alpha }}dxdy  =\int_{{\rm int}( K_{f})}\int_{S^{n-1}}\int_{\rho_{K_f}(u)}^{\rho_{K_{\widehat{f}_t}}(u)}\frac{\widehat{f}_t(ru)\widehat{f}_t(y)}{t|ru-y|^{n-\alpha }}r^{n-1}drdudy,
\end{align} where the integral is treated as $0$ on $\Omega _{f}^c.$

Since  $\widehat{f}_t\in \LC$, by Proposition \ref{finiteness of I_q}, the Fubini-Tonelli theorem, and the role switching of $x$ and $y$ (in the second equality), one gets
\begin{align*} 
\int_{K_{f}}  \int_{ K_f}\frac{[\widehat{f}_t(x) - f(x)]f(y) }{t|x-y| ^{ n-\alpha }}dx dy&=\int_{K_{f}}  \int_{ K_f}\frac{[\widehat{f}_t(x) - f(x)]f(y) }{t|x-y| ^{ n-\alpha }}dy dx\\
&=\int_{K_{f}}\int_{ K_f}\frac{\widehat{f}_t(y)f(x)- f(y)f(x)}{ t|x-y| ^{ n-\alpha }}dxdy\\
&=\int_{{\rm int} (K_{f})}\int_{ K_f}\frac{\widehat{f}_t(y)f(x)- f(y)f(x)}{t|x-y| ^{ n-\alpha }}dxdy.\end{align*} 
This further implies that 
\begin{align}\label{TF0}  \nonumber
&  \int_{{\rm int} (K_{f})}\int_{ K_f}\frac{\widehat{f}_t(x)\widehat{f}_t(y)-f(x)f(y)}{2t|x-y| ^{ n-\alpha }}dxdy\\ 
&\quad =   \int_{{\rm int} (K_{f})}\int_{ K_f}\frac{\widehat{f}_t(x)\widehat{f}_t(y)- \widehat{f}_t(x)f(y)}{2t|x-y| ^{ n-\alpha }}dxdy +    \int_{{\rm int}(K_{f})}\int_{ K_f}\frac{\widehat{f}_t(x)f(y)- f(x)f(y)}{2t|x-y| ^{ n-\alpha }}dxdy\nonumber \\
&\quad =  \int_{{\rm int} (K_{f})}\int_{K_f}\frac{\widehat{f}_t(x)\widehat{f}_t(y)- \widehat{f}_t(x)f(y)}{2t|x-y| ^{ n-\alpha }}dxdy +    \int_{{\rm int}(K_{f})}\int_{ K_f}\frac{\widehat{f}_t(y)f(x)- f(y)f(x)}{2t|x-y| ^{ n-\alpha }}dxdy \nonumber\\
&\quad =\!\frac{1}{2} \int_{{\rm int}(K_{f})} \!\int_{ S^{n-1}}\!\int_0^{\rho_{K_f}(u)}\!\!\left(\frac{\widehat{f}_t(r u)\left[\widehat{f}_t(y)-  f(y)\right]}{t|r u-y| ^{ n-\alpha }}+\frac{\left[\widehat{f}_t(y) - f(y)\right]f(r u) }{t|r u-y| ^{ n-\alpha }} \right) r^{n-1} drdudy  .
\end{align}
From  \eqref{int-out0-1}, \eqref{int} and \eqref{TF0},  one has 
\begin{align}\label{int-out1}
 \delta \Ia(f,\mathfrak{g})  =\lim_{t\to 0^+} \int_{{\rm int}(K_{f})}\int_{\sphere} E_t(u,y) dudy.
\end{align}

Next, we calculate    $\lim_{t\to 0^{+}}E_{t }(u,y)$. For any fixed $y\in{\rm int}(K_f)$ and $u\in \Omega_f$,   from   the  mean value theorem for  the definite integrals,  there exists $r_0(t):=r_0(t,u,y)\in (1,1+\beta_1 t)$   such that 
\begin{align}\label{E-01}
 \lim_{t\to0^{+}}\int_{\rho_{K_f}(u)}^{\rho_{K_{\widehat{f}_t}}(u)}\frac{\widehat{f}_t(ru)\widehat{f}_t(y)}{t|r u-y|^{ n-\alpha}}r^{n-1}dr&=\rho_{K_f}(u)^n   \lim_{t\to0^{+}}\int_{1}^{1+t\beta_1}\frac{\widehat{f}_t(r\rho_{K_f}(u)u)\widehat{f}_t(y)}{t|r\rho_{K_f}(u)u-y|^{ n-\alpha }}r^{n-1}dr\nonumber\\
& = \rho_{K_f}(u)^n\lim_{t\to0^{+}}\beta_1\frac{{\widehat{f}}_t(r_0(t)\rho_{K_{f}}(u)u){\widehat{f}}_t(y)}{|r_0(t)\rho_{K_f}(u)u-y|^{ n-\alpha }}r_0(t)^{n-1}\nonumber\\
&= \beta_1\rho_{K_f}(u)^n\frac{  f(\rho_{K_{f}}(u)u)f(y)}{|\rho_{K_f}(u)u-y|^{n-\alpha}},
\end{align} 
where we have used \eqref{hat-f-to0}.   On the other hand,  for almost all $y\in {\rm int}(K_f)$,  \eqref{rotem-formula1}  implies that 
\begin{align*}
\lim_{t\to0^{+}} \int_0^{\rho_{K_f}(u)}\frac{(\widehat{f}_t(y)\!-\!  f(y)) \widehat{f}_t(r u)\, r^{n-1}  }{2t|r u-y| ^{ n-\alpha }}dr 
& \! =\!\lim_{t\to0^{+}}\!\frac{\widehat{f}_t(y)\!-\!f(y)}{2t}\times \int_{0}^{\rho_{K_f}(u)} \!\!\! \frac{\widehat{f}_t(ru)  r^{n-1}}{ |ru-y|^{ n-\alpha }} dr\nonumber\\
& \! =\! \frac{ (\beta_2\!+\! \beta_1\varphi^*(\nabla\varphi(y))) \, f(y)}{2}\times\! \lim_{t\to 0^+}\int_{0}^{\rho_{K_f}\!(u)}\!\! \frac{\widehat{f}_t(ru) r^{n-1}}{ |ru\!-\!y|^{ n-\alpha }} dr.
\end{align*}
Since $\widehat{f}_t\in{\rm LC}_n$, by \eqref{In.CF13-11}, one can find two constants $b>0$ and $c\in\mathbb{R}$ such that, for all $t\in [0, 1]$,  
$\widehat{f}_t(x)\le c_0 e^{-b|x|},$ where $c_0\!=\! e^{|c|+|\beta_2-\beta_1c|}.$  Consequently, for  all $u\!\in\! \sphere$ and for all $t\!\in\! [0, 1]$, one gets $$ \frac{\widehat{f}_t(ru)   }{ |ru-y|^{ n-\alpha }} r^{n-1} \leq \frac{c_0 e^{-br}}{ |ru-y|^{ n-\alpha }}r^{n-1}.$$
According to Proposition \ref{estimate-riesz potential},
one gets
\begin{align*}
I_\alpha(c_0 e^{-b|\cdot|},y)
=c_0 \int_{\mathbb{R}^n}\frac{e^{-b|x|}}{|x-y|^{n-\alpha}}dx
=c_0 \int_{S^{n-1}}\int_{0}^\infty\frac{e^{-br}}{|ru-y|^{n-\alpha}}r^{n-1}drdu<\infty.
\end{align*}
So, for almost all $u\in S^{n-1}$, one has
\begin{align}\label{DCT-req-1}
\int_{0}^{\infty}\frac{e^{-br} }{ |ru-y|^{ n-\alpha }} r^{n-1}dr<\infty.
\end{align}   It follows from \eqref{hat-f-to0}, \eqref{DCT-req-1} and the dominated convergence theorem that  
\begin{align*}\lim_{t\to0^{+}}\!  \int_{0}^{\rho_{K_f}(u)}\!\! \! \frac{\widehat{f}_t(ru)}{|ru\!-\!y|^{ n-\alpha }} r^{n-1}dr =\! \int_{0}^{\rho_{K_f}(u)}\!\!\! \lim_{t\to0^{+}} \frac{\widehat{f}_t(ru)}{|ru\!-\!y|^{ n-\alpha }} r^{n-1}dr =\! \int_0^{\rho_{K_f}(u)} \!\!\! \frac{f(ru)}{ |ru\!-\!y|^{ n-\alpha }} r^{n-1}dr
\end{align*}
for almost all $u\in S^{n-1}$.
Hence, for almost all $u\in S^{n-1}$ and $y\in {\rm int}(K_f)$,
\begin{align}
\lim_{t\to0^{+}}\! \int_0^{\rho_{K_f}\!(u)}\!\frac{( \widehat{f}_t(y)\!-\!\! f(y))\widehat{f}_t(r u)\, r^{n-1}}{2t|r u-y| ^{ n-\alpha }} dr \!=\!\frac{ (\beta_2\!+\!\beta_1\varphi^*(\nabla\varphi(y))) f(y)}{2}\times \! \int_0^{\rho_{K_f}(u)}\!\! \frac{f(ru) r^{n-1}}{ |ru\!-\!y|^{ n-\alpha }}dr. \label{E-02}
\end{align}
Similarly,  for almost all $u\in S^{n-1}$ and $y\in {\rm int}(K_f)$, one has
\begin{align}\label{E-03}
\lim_{t\to0^{+}}\! \int_0^{\rho_{K_f}\!(u)}\!\frac{( \widehat{f}_t(y)\!-\!\! f(y))f(r u)\, r^{n-1}}{2t|r u-y| ^{ n-\alpha }} dr \!=\!\frac{ (\beta_2\!+\!\beta_1\varphi^*(\nabla\varphi(y))) f(y)}{2}\times \! \int_0^{\rho_{K_f}(u)}\!\! \frac{f(ru) r^{n-1}}{ |ru\!-\!y|^{ n-\alpha }}dr.\end{align}
Combining \eqref{E-01}, \eqref{E-02} with \eqref{E-03}, one has
\begin{align*}
\lim_{t\to0^{+}}E_t(u,y)&=\beta_1\frac{  f(\rho_{K_{f}}(u)u)f(y)}{|\rho_{K_f}(u)u-y|^{n-\alpha}} \rho_{K_f}(u)^{n}+   \int_0^{\rho_{K_f}(u)}\frac{(\beta_2+\beta_1\varphi^*(\nabla\varphi(y)))f(ru) f(y)  }{ |ru-y| ^{ n-\alpha }} r^{n-1}dr,
\end{align*}
for  almost all  $y\in {\rm int}(K_f)$ and $u\in \sphere$.  Consequently, from Propositon \ref{Le.ff-2}, one has
 \begin{align}\label{int-in0}
& \int_{K_{f}}\int_{\sphere} \lim_{t\to0^{+}} E_{t }(u,y)dudy \nonumber\\
&\quad =\beta_1 \int_{K_{f}}\int_{\sphere}\frac{  f(\rho_{K_{f}}(u)u)f(y)}{|\rho_{K_{f}}(u)u-y|^{n-\alpha}}\rho_{K_f}(u)^{n}dudy \nonumber\\
&\quad\quad+ \int_{K_{f}}\int_{\sphere} \int_0^{\rho_{K_f}(u)}\frac{(\beta_2+\beta_1\varphi^*(\nabla\varphi(y))) f(ru)f(y)}{ |ru-y|^{n-\alpha}} r^{n-1}drdudy \nonumber\\
&\quad =\beta_1\int_{K_{f}}\!\int_{ \partial K_f}\!\!\!\frac{\langle x,\nu_{K_f}(x)\rangle  f(x)f(y)}{|x-y|^{n-\alpha}} d\mathcal{H}^{n-1}(x)dy+\beta_1  \int_{K_f}\! \int_{K_f}\!\!  \frac{ \varphi^*(\nabla\varphi(y))f(x)f(y)}{|x-y|^{n-\alpha}} dxdy\!+\!\beta_2\Ia(f)\nonumber \\
&\quad= \beta_1 \delta \Ia(f,f)+\beta_2 \Ia(f),
\end{align} where we have used the polar coordinates (for the integral over $K_f\times K_f$) and the variable change $\rho_{K_f}\!(u)u\!=\! x$ (for the integral over $K_f \!\times\! \partial K_f$). The desired equality follows  from \eqref{formula 4.6},  \eqref{int-out1} and \eqref{int-in0}. 
\end{proof}

 For any convex bodies $K,L$ containing  the origin in their interiors, the following  has been  proved in \cite{HLYZ16}
\begin{eqnarray}\label{radial-diff}
\lim_{t\to0^{+}}\frac{\rho_{K+tL}(u)-\rho_{K}(u)}{t}=\frac{h_{L}(\nu_{K}(\rho_{K}(u)u))}{\langle u, \nu_{K}(\rho_{K}(u)u)\rangle} ,
\end{eqnarray}
for  almost all  $u\in S^{n-1}$. Here $\rho_K(u)=\sup\{\lambda:\lambda u\in K\}$  is the radial function of $K$. In fact,    \eqref{radial-diff}  also holds when $K$ and $L$ are unbounded but with $\rho_K$ and $\rho_L$ being finite.

\begin{lemma}\label{radial-diff-lem}
Let $f\!=\!e^{-\varphi}\!\in\!\LC$ with $o\!\in\!{\rm int}(K_f)$ and $g\!=\!e^{-\psi}\!\in\!\LC$.  If $\varphi$ and $\psi$ satisfy \eqref{nat-condition-1}, then 
\begin{align*}
\lim_{t\to0^{+}}\frac{\rho_{K_{f_t}}(u)-\rho_{K_f}(u)}{t}=\frac{h_{K_g}(\nu_{K_f}(\rho_{K_f}(u)u))}{\langle u, \nu_{K_f}(\rho_{K_f}(u)u)\rangle} ,
\end{align*}
for  almost all  $u\in \Omega_f$. Here $f_t=f\oplus (t\bullet g)$.
\end{lemma}
\begin{proof}
For any $N\in\mathbb{N}$, we set 
$\widehat{K}_{f,N}:=K_f\cap  B_2^n(o,N)$, 
which is a convex body containing the origin in its interior. Let $\sigma_{f,N}$ be the subset of $\Omega_f$ such that,  if $u\in \sigma_{f,N}$, then $\rho_{K_f}(u)u$ has more than one outer normal vector.  Indeed, for each $N\in \mathbb{N}$, $\sigma_{f,N}$ is of spherical Lebesgue  measure zero. Let $\eta_{f}=\bigcup_{N\in\mathbb{N}} \sigma_{f,N}$, and then     $$0\leq \mathcal{H}^{n-1}(\eta_{f})\leq \sum_{N=1}^{\infty}\mathcal{H}^{n-1}(\sigma_{f,N})=0,$$ 
 i.e., the set  $\eta_{f} \subset \Omega_f$  is of   $\mathcal{H}^{n-1}$-measure zero.

  Let $u\in \Omega_f\setminus \eta_f$.  Then  there is only one normal vector at $\rho_{K_f}(u)u$. Let $v_t$ be any normal vector of $K_{f_t}$ at the point $\rho_{K_{f_t}}(u)u\in \partial K_{f_t}$.  One can get 
 $$\rho_{K_{f_t}}(u)=\frac{h_{K_{f_t}}(v_t)}{\langle u,v_t\rangle } \quad \text{and}\quad \rho_{K_{f}}(u)\leq \frac{h_{K_{f}}(v_t)}{\langle u,v_t\rangle }.$$ 
Note that $h_{K_{f_t}}(v_t)\leq\rho_{K_{f_t}}(u)\leq (1+\beta_1t) \rho_{K_{f}}(u)<+\infty$.  It follows from  \cite[Proposition 2.1]{CF13} that $K_{f_t}\!\!=\!\!K_f\!+\! tK_g$ and thus $h_{K_{f_t}}\!\!=\! h_{K_f}\!+\!t h_{K_g}$, whenever $h_{K_f}$ and $h_{K_g}$ are  finite. This further gives  
\begin{align}\label{supp-1}
    \frac{\rho_{K_{f_t}}(u)-\rho_{K_{f}}(u)}{t}\geq \frac{h_{K_{f_t}}(v_t)-h_{K_{f}}(v_t)}{\langle u,v_t\rangle  t}=\frac{h_{K_{g}}(v_t)}{\langle u,v_t\rangle}.
\end{align}
Similarly, if we let $v\in S^{n-1}$ (which is unique) be such that $h_{K_f}(v)=\langle \rho_{K_f}(u)u,v\rangle$, then  
\begin{align}\label{supp-2}
    \frac{\rho_{K_{f_t}}(u)-\rho_{K_{f}}(u)}{t}\leq \frac{h_{K_{g}}(v)}{\langle u,v\rangle}.
\end{align} We need to show that $v_t\to v$  as $t\to0^{+}$.  To this end,  for any $t_k\to 0^{+}$, we  can pick a  convergent subsequence  of $v_{t_k}$ (due to the compactness of $\sphere$), which will be still denoted by $v_{t_k}$, such that $v_{t_{k}}\to v'$ as $t_k\to 0^{+}$. Since $K_f\subseteq K_{f_{t_k}}\subseteq (1+\beta_1 t_k)K_f$, one gets $\rho_{K_{f_{t_k}}}(u)\rightarrow \rho_{K_{f}}(u)$ and 
\begin{align}
    \lim_{k\rightarrow \infty} h_{K_{f_{t_k}}}(v_{t_k})= \lim_{k\rightarrow \infty} \langle \rho_{K_{f_{t_k}}}(u)u,v_{t_k}\rangle =\langle \rho_{K_{f}}(u)u,v'\rangle. \label{new-5-32}
\end{align}

On the other hand, one can take $R_0=2(1+\beta_1)^2\rho_{K_f}(u)$, and then $h_{K_{f_t}}(v_t)\leq \frac{R_0}{2}$ for all $t\in [0, 1]$. In particular,  for all $k\in \mathbb{N}$,   $$h_{K_{f}\cap B_2^n(o, R_0)}(v_{t_k})=h_{K_f}(v_{t_k}) \leq h_{K_{f_{t_k}}}(v_{t_k}) \leq \frac{R_0}{2}.$$   Note that the support function of the convex body  $K_{f}\cap B_2^n(o, R_0)$ is continuous on $\sphere$. Therefore, 
$$\lim_{k\rightarrow \infty}  h_{K_{f}}(v_{t_k})=\lim_{k\rightarrow \infty} h_{K_{f}\cap B_2^n(o, R_0)}(v_{t_k})= h_{K_{f}\cap B_2^n(o, R_0)}(v')\leq \frac{R_0}{2},$$ which further implies 
$$\lim_{k\rightarrow \infty}  h_{K_{f}}(v_{t_k}) = h_{K_{f}\cap B_2^n(o, R_0)}(v') = h_{K_{f}}(v'). $$ Clearly,  $h_{K_{g}}(v_{t_k})\!\leq\! \beta_1 h_{K_{f}}(v_{t_k}) \!\leq \! \beta_1  h_{K_{f_{t_k}}}(v_{t_k})\!\leq\! \beta_1 (1+\beta_1) \rho_{K_f}(u)\!<\!\infty$ for all $k\in \mathbb{N}$. Consequently,  $$ \lim_{k\rightarrow \infty}  h_{K_{f_{t_k}}}(v_{t_k})=\lim_{k\rightarrow \infty}  h_{K_{f}}(v_{t_k})+\lim_{k\rightarrow \infty}  t_k h_{K_{g}}(v_{t_k})=h_{K_f}(v').$$ 
Together with \eqref{new-5-32}, one gets $h_{K_f}(v')
=\langle \rho_{K_{f}}(u)u,v'\rangle.$ Hence $v'$ is a normal vector of $\partial K_f$ at $\rho_{K_f}(u)u$. Due to the uniqueness of the normal vector of $\partial K_f$ at $\rho_{K_{f}}(u)u$, one gets $v=v'$. So  $v_{t_k}\to v$  as $t_k\to0^{+}$ and thus $v_t\rightarrow v$ as $t\to 0^+$. Combining with \eqref{supp-1} and \eqref{supp-2},  we can obtain the desired formula. 
\end{proof}

All auxiliary lemmas have been completed,  and  we now  can prove Theorem
\ref{var.bound.-1}.

\begin{proof}[Proof of Theorem \ref{var.bound.-1}] Let $f_t=f\oplus (t\bullet g)$ and $\widehat{f}_t=f\oplus (t\bullet e^{-(\beta_1 \varphi^*+\beta_2)^*)})$  for some $\beta_1>0$ and $\beta_2\in\R$. Condition \eqref{nat-condition-1} implies that  $f_t\leq \widehat{f}_t$  and $K_{f_t}\subseteq K_{\widehat{f}_t}$.  Lemma \ref{Le.J1} deduces that
\begin{align*}
0=\lim_{t\to 0^+}\int_{K_{\widehat{f}_t}\setminus K_f}\int_{K_{\widehat{f}_t}\setminus K_f}\frac{\widehat{f}_t(x)\widehat{f}_t(y)}{t|x-y|^{n-\alpha}}dxdy\geq\lim_{t\to 0^+}\int_{K_{f_t}\setminus K_f}\int_{K_{f_t}\setminus K_f}\frac{f_t(x)f_t(y)}{t|x-y|^{n-\alpha}}dxdy \geq 0. 
\end{align*} In particular, we have \begin{align*} 
 \lim_{t\to 0^+}\int_{K_{f_t}\setminus K_f}\int_{K_{f_t}\setminus K_f}\frac{f_t(x)f_t(y)}{t|x-y|^{n-\alpha}}dxdy = 0.
\end{align*}

  For any $t>0$, $u\in S^{n-1}$ and $y\in {\rm int }(K_f)$, we set
  \begin{align*}
F_{t }(u,y)=\int_{\rho_{K_f}(u)}^{\rho_{K_{f_t}}(u)}\!\frac{f_t(r u)f_t(y)}{t|r u-y|^{ n-\alpha}}r^{n-1}dr+\!\int_0^{\rho_{K_f}(u)}\! \frac{f_t(r u)\left[f_t(y)\!-\!  f(y)\right]+\left[f_t(y)\!-\!f(y)\right]f(r u)}{2t|r u-y| ^{n-\alpha }} r^{n-1}dr. 
\end{align*} Again, if $u\in\Omega_f^c$, one views $$\int_{\rho_{K_f}(u)}^{\rho_{K_{f_t}}(u)}\frac{f_t(r u)f_t(y)}{t|r u-y|^{ n-\alpha}}r^{n-1}dr=0.$$  

First, we consider the case when $\psi(o)=0$, where $\psi=-\log g$. In this case,  $f\leq f_t$ for any $t\in[0,1]$ (see e.g.,   \cite[Lemma 3.9]{CF13}). Moreover, $f_t\leq \widehat{f}_t$ for any $t\in [0, 1]$ due to condition \eqref{nat-condition-1}. Hence,    $0\le F_t(u,y)\le E_t(u,y)$ for any $u\in S^{n-1}$,  $y\in {\rm int}(K_f)$ and $t\in(0,1)$, where $E_t(\cdot,\cdot)$  is  given  in \eqref{E_t}. Since $f_t,\widehat{f}_t\in{\rm LC}_n$,  Proposition \ref{finiteness of I_q}  and \eqref{int-out1} ensure that
 both  $E_t(u,y)$ and $F_t(u,y)$  are integrable on   $\sphere\times K_f$ for any $t>0$. When $t\to 0^+$, by Lemma \ref{radial-diff-lem} and \eqref{rotem-formula1}, one has, for almost all $u\in\sphere$ and $y\in K_f,$
\begin{align*}
\lim_{t\to0^+}\!E_t(u,y)&\geq  \lim_{t\to 0^+}F_t(u,y)\\ 
&= \frac{ h_{K_g}(\nu_{K_f}\!(\rho_{K_f}\!(u)u))\rho_{K_f}\!(u)^{n-\!1}\! f(\rho_{K_{f}}\!(u)u)f(y)}{|\rho_{K_f}\!(u)u-y|^{n-\alpha}\langle u, \nu_{K_f}(\rho_{K_f}(u)u)\rangle} \!+\! \psi^*(\nabla\varphi(y))f(y) \!\! \int_0^{\rho_{K_f}\!(u)}\!\! \frac{f(ru)  r^{n-1} }{ |r u\!-\!y|^{ n-\alpha}}dr,
\end{align*}
where, if $u\in \Omega_f^c$, the first term in the right hands side of the equality is treated as $0$.
Together with Lemma \ref{Le.Iff-bound.}, the general dominated convergence theorem can be applied to get
\begin{align}\label{lim-F}
&\lim_{t\to 0^+}\int_{K_f}\int_{S^{n-1}}F_t(u,y)dxdy =\int_{K_f}\int_{S^{n-1}}\lim_{t\to 0^+}F_t(x,y)dxdy\nonumber\\
&\quad=\int_{K_f}\int_{S^{n-1}} \int_{\rho_{K_f}(u)}^{\rho_{K_{f_t}}(u)}\!\frac{f_t(r u)f_t(y)}{t|r u-y|^{ n-\alpha}}r^{n-1}dr\,du\,dy \nonumber \\ &\quad\quad + \int_{K_f}\int_{S^{n-1}} \!\int_0^{\rho_{K_f}(u)}\! \frac{f_t(r u)\left[f_t(y)\!-\!  f(y)\right]+\left[f_t(y)\!-\!f(y)\right]f(r u)}{2t|r u-y| ^{n-\alpha }} r^{n-1}drdudy
\nonumber\\
&\quad=\int_{K_f}\int_{\partial K_f} \frac{h_{K_g}(\nu_{K_f}(x))  f(x)f(y)}{|x-y|^{n-\alpha}}d\mathcal{H}^{n-1}(x)dy+\int_{K_f}\int_{K_f}\frac{\psi^\ast(\nabla\varphi(y))f(x)f(y)}{|x-y|^{n-\alpha}}dxdy,
\end{align} where we have used the polar coordinates (for the integral over $K_f\times K_f$) and the variable change $\rho_{K_f}(u)u=x$ (for the integral over $K_f\times \partial K_f$). 

Along the same lines of the proof for \eqref{int-out1}, one can get 
\begin{align*}
\delta \Ia(f,g)
&=\frac{1}{2}\lim_{t\to 0^+}\int_{K_{f_t}}\int_{K_{f_t}}\frac{f_t(x)f_t(y)-f(x)f(y)}{t|x-y|^{n-\alpha}}dxdy\\
&=\frac{1}{2}\lim_{t\to 0^+}\int_{K_{f_t}\setminus K_f}\int_{K_{f_t}\setminus K_f}\frac{f_t(x)f_t(y)}{t|x-y|^{n-\alpha}}dxdy+\frac{1}{2}\lim_{t\to 0^+}\int_{K_{f}}\int_{K_{f_t}\setminus K_f}\frac{f_t(x)f_t(y)}{t|x-y|^{n-\alpha}}dxdy\\
&\quad+\frac{1}{2}\lim_{t\to 0^+}\int_{K_{f_t}\setminus K_f}\int_{K_{f}}\frac{f_t(x)f_t(y)}{t|x-y|^{n-\alpha}}dxdy +\frac{1}{2}\lim_{t\to 0^+}\int_{K_f}\int_{K_{f}}\frac{f_t(x)f_t(y)-f(x)f(y)}{t|x-y|^{n-\alpha}}dxdy\\
&=\lim_{t\to 0^+}\int_{K_{f}}\int_{K_{f_t}\setminus K_f}\frac{f_t(x)f_t(y)}{t|x-y|^{n-\alpha}}dxdy +\frac{1}{2}\lim_{t\to 0^+}\int_{K_f}\int_{K_{f}}\frac{f_t(x)f_t(y)-f(x)f(y)}{t|x-y|^{n-\alpha}}dxdy\\
&=\lim_{t\to 0^+}\int_{K_f}\int_{S^{n-1}}F_t(u,y)dxdy.
\end{align*}
The desired formula follows from  \eqref{lim-F} and the Fubini theorem, namely, \begin{align*}
 \delta \Ia(f,g) &=\lim_{t\to 0^+}\int_{K_f}\int_{S^{n-1}}F_t(u,y)dxdy\\  &=\int_{K_f}\int_{S^{n-1}} \lim_{t\to 0^+} F_t(u,y)dxdy \\ &=\int_{K_f}\int_{\partial K_f} \frac{h_{K_g}(\nu_{K_f}(x))  f(x)f(y)}{|x-y|^{n-\alpha}}d\mathcal{H}^{n-1}(x)dy+\int_{K_f}\int_{K_f}\frac{\psi^\ast(\nabla\varphi(y))f(x)f(y)}{|x-y|^{n-\alpha}}dxdy\\ &=\int_{\partial K_f}\int_{K_f} \frac{h_{K_g}(\nu_{K_f}(x))  f(x)f(y)}{|x-y|^{n-\alpha}}dyd\mathcal{H}^{n-1}(x)+\int_{K_f}\int_{K_f}\frac{\psi^\ast(\nabla\varphi(x))f(x)f(y)}{|x-y|^{n-\alpha}}dxdy. 
\end{align*} This completes the proof of Theorem \ref{var.bound.-1} when $\psi(o)=0$.    

In general, if $\psi(o)\neq0$, we set $\widetilde{\psi}=\psi-\psi(o)$ and $\widetilde{f}_t=f\oplus (t\bullet \widetilde{g})$ where $\widetilde{g}=e^{-\widetilde{\psi}}$. Then $\widetilde{\psi}(o)=0$. By  \eqref{def-dual} and \eqref{alpha-energy-def}, one has  
\begin{align*} 
\delta \Ia(f,\widetilde{g})&=\!\!  \int_{\partial K_f}\! \int_{K_f}  \!\!\! \!\!  \frac{h_{K_g}(\nu_{K_f}(x))  f(x)f(y)}{|x-y|^{n-\alpha}}dyd\mathcal{H}^{n-1}(x)\!+\!\! \int_{K_f}\!\int_{K_f} \!\!\! \! \frac{\widetilde{\psi}^\ast(\nabla\varphi(x))f(x)f(y)}{|x-y|^{n-\alpha}}dxdy\\
&=\!\!  \int_{\partial K_f}\! \int_{K_f} \!\!\! \!\! \frac{h_{K_g}\!(\nu_{K_f}\!(x))  f(x)f(y)}{|x-y|^{n-\alpha}}dyd\mathcal{H}^{n-1}(x)\!+\!\! \int_{K_f}\! \int_{K_f} \!\!\! \!\frac{\psi^\ast(\nabla\varphi(x))f(x)f(y)}{|x-y|^{n-\alpha}}dxdy \! +\! \psi(o)  \Ia(f).
\end{align*}
From \eqref{add-1}, one  has $  
\delta \Ia(f,\widetilde{g})=  \delta \Ia(f,g)+\psi(o)  \Ia(f),$ which immediately yields the desired result in Theorem \ref{var.bound.-1}. 
\end{proof}

\section{The Riesz \texorpdfstring{$\alpha$}{}-energy Minkowski problem}
\label{section:-6}

Theorem \ref{var.bound.-1}  motivates the definition of a measure for log-concave functions $f\in \LC$ defined on $\Rn$, as explained in \eqref{def-Rma-22-1}. 

\begin{definition}
Let $f=e^{-\varphi}\in\LC$ and $\alpha>0$. Define, $\Rma(f,\cdot)$, the 
Riesz $\alpha$-energy measure of $f$, as the push-forward of the measure $ I_{\alpha}(f, y) f(y) dy$ under $\nabla \varphi$, i.e., 
\begin{align}\label{def-Rma-22}
\int_{\mathbb{R}^n}\mathbf{h}(z)d\Rma(f,z)=\int_{K_f}\int_{K_f} \frac{\mathbf{h}(\nabla\varphi(y))f(x)f(y)}{|x-y|^{n-\alpha}}dxdy =  \int_{K_f}  \mathbf{h}(\nabla\varphi(y)) I_{\alpha}(f, y) f(y) dy, 
\end{align}
for each Borel function $\mathbf{h}:\mathbb{R}^n\to\mathbb{R}$ such that $\mathbf{h}\in\mathcal{L}^1(\Rma(f,\cdot))$ or $h$ is non-negative.
\end{definition} 
  Clearly, for any constant $c>0$, \begin{align}
    \label{Rma-homogeneous} d\Rma(cf,\cdot) =c^2d\Rma(f,\cdot).
\end{align}

Recall the following nature question  regarding the measure $\Rma(f,\cdot)$. \vskip 2mm \noindent{\bf Problem \ref{Chord-Mink-Pro-1} (The Riesz $\alpha$-energy Minkowski problem for log-concave functions).} {\em 
 Let $\mu$ be a finite Borel measure on $\Rn$ and $\alpha>0$. Find the necessary and/or sufficient conditions on $\mu$ so that $
    \mu=\Rma(f,\cdot)$
holds for some log-concave function $f\in\LC$.}

\vskip 2mm  Recall that \eqref{RN} holds, i.e., assuming  enough regularities,   
\begin{align*}
    \frac{\,d\Rma(f, z)}{dz}=I_{\alpha}(f,\nabla \varphi^*(z)) f(\nabla \varphi^*(z))\det (\nabla^2\varphi^*(z)). 
\end{align*} Consequently, Problem \ref{Chord-Mink-Pro-1} reduces to the following Monge-Amp\`{e}re type partial differential equation, with $\varphi$ the unknown convex function: 
\begin{align*}
g(z)=I_{\alpha}(f, \nabla \varphi^*(z)) f(\nabla \varphi^*(z))\det (\nabla^2\varphi^*(z)).
\end{align*}

In this section, we provide a solution to Problem \ref{Chord-Mink-Pro-1} for even data. Let $\mu$ be an even finite Borel measure on $\R^n$. Consider the following  optimization problem: for $\alpha>0,$
\begin{align}\label{optimiza-pro}
 \Theta_{\tau}=\inf\left\{\int_{\mathbb{R}^n}\varphi d\mu:
    \varphi\in \mathcal{L}_e^+(\mu)\ \ \text{and}\ \  \Ia(e^{-\varphi^\ast})\ge \tau \right\},
\end{align} where  $\tau>0$ is a pre-given constant and  $\mathcal{L}_e^+(\mu)$ is the class of all even non-negative integrable functions with respect to   $\mu$. Let $\C(\R^n) $ denote the class of lower semi-continuous convex functions from $\R^n$  to $\R\cup\{+\infty\}$.  For any $\varphi \in \mathcal{L}_e^+(\mu)$, it is clear that   $\varphi\geq (\varphi^*)^*$ and $\varphi^*=((\varphi^*)^*)^*$, which imply, for $\alpha>0,$ 
\begin{align*}
\int_{\mathbb{R}^n}\varphi d\mu\geq \int_{\mathbb{R}^n}(\varphi^*)^*d\mu \quad\text{and}\quad \Ia(e^{-\varphi^\ast})= \Ia(e^{-((\varphi^*)^*)^*}).
\end{align*}
Therefore,  to solve  the optimization problem  \eqref{optimiza-pro} is  equivalent to solving the following problem: for $\alpha>0,$
\begin{align}\label{optimiza-convex}
  \Theta_{\tau}=\inf\left\{\int_{\mathbb{R}^n}\varphi d\mu:
    \varphi\in \mathcal{L}_e^+(\mu) \cap \C(\R^n) \ \ \text{and}\ \  \Ia(e^{-\varphi^\ast})\ge  \tau \right\}.
\end{align}

We will explain that the existence of solutions to the Riesz $\alpha$-energy Minkowski problem for log-concave functions (i.e., Problem \ref{Chord-Mink-Pro-1})  can be established by solving the optimization problem \eqref{optimiza-convex}. To this end, we first show that the infimum in  \eqref{optimiza-convex} is attained when the given Borel measure $\mu$ is even. The natural condition on the pre-given measure $\mu$ is  $\mu\in \mathscr{M}$, where $\mathscr{M}$ is the set of  all  even nonzero finite Borel measures on $\R^n$  satisfying that $\mu$ is not concentrated in any subspace and the first moment of $\mu$ is finite.  Equivalently, if $\mu\in \mathscr{M}$, then there exists a constant $c_\mu>0$ (depending on the measure $\mu$ only) such that
\begin{align}\label{not-concentrated}
&\int_{\mathbb{R}^n}|\langle x, u\rangle|d\mu>c_\mu \ \ \mathrm{for\ any} \ \ u\in \sphere; 
\\ \label{first-moment}
&\int_{\mathbb{R}^n}|x|d\mu<\infty.
\end{align} 

We now state  the existence of optimizers for  \eqref{optimiza-convex}.
\begin{proposition}\label{existence-of-function}
Let $\mu\in \mathscr{M}$. There exists    $\varphi\in  \mathcal{L}_e^+(\mu) \cap \C(\R^n)$ solving the optimization problem \eqref{optimiza-convex}.  Moreover, $\varphi$ is strictly positive on $\mathbb{R}^n$ and $\Ia(e^{-\varphi^\ast})=\tau$ when $\tau$ is big enough.
\end{proposition}

The proof of Proposition \ref{existence-of-function} relies on a few lemmas.
\begin{lemma}\label{Le.roc} \cite[Theorem 10.9]{Roc70}
    Let $\Omega\subset\mathbb{R}^n$ be an open convex set and $\{f_i\}_{i\in\mathbb{N}}$ be a   sequence of finite convex functions on $\Omega$. Suppose that the real number sequence $\{f_i(x)\}_{i\in\mathbb{N}}$ is bounded for each $x\in \Omega$. It is then possible to select a subsequence of $\{f_i\}_{i\in\mathbb{N}}$, which  converges to a finite convex function $f$ pointwisely on $\Omega$ and uniformly on any closed
bounded subset of $\Omega$.
\end{lemma}

For $E\subset \Rn,$ let  ${\rm conv} E$ be the closed convex hull of $E$, namely the smallest closed convex set containing $E$. For $\mu\in\mathscr{M}$, the support of $\mu$ is denoted by ${\rm Supp}(\mu)$, and let $$M_\mu=\mathrm{int}({\rm conv}({\rm Supp}(\mu))).$$ 
 For $x_0\in  M_\mu$, by \cite[Lemma 16]{CK2015}, one can find a constant $C_{\mu,x_0}$ so that   for any $\varphi\in  \mathcal{L}_e^+(\mu) \cap \C(\R^n)$,
\begin{align}\label{finite-varphi(x0)}
    \varphi(x_0)\le
    C_{\mu,x_0}\int_{\mathbb{R}^n}\varphi d\mu.
\end{align}
Using  \eqref{finite-varphi(x0)} and Lemma \ref{Le.roc},  we can show that  the infimum in  \eqref{optimiza-convex} is attained at non-negative  convex functions.  We need to remark that the method  we used is inspired by \cite[Lemma 17]{CK2015}. To keep the content complete, we will give its proof, and but keep it brief. 
 
\begin{lemma}\label{Le.existence}
Let $\alpha>0$, $\mu\in\mathscr{M}$ and $\{\varphi_i\}_{i\in\mathbb{N}}$ be a sequence of functions  with  $\varphi_i\in \mathcal{L}_e^+(\mu) \cap \C(\R^n)$ for any $i\in\mathbb{N}$. Assume that
$$
\sup_{i\in\mathbb{N}}\int_{\mathbb{R}^n}\varphi_id\mu<\infty.
$$
Then there exists a subsequence $\{\varphi_{i_j}\}_{j\in\mathbb{N}}$ and a function $\varphi\in  \mathcal{L}_e^+(\mu) \cap \C(\R^n)$ such that
\begin{align}\label{two-limits}
    \int_{\mathbb{R}^n}\varphi d\mu\le\liminf_{j\to\infty}\int_{\mathbb{R}^n}\varphi_{i_j} d\mu
    \ \ \mathrm{and}\ \ \
    \Ia(e^{-\varphi^\ast})\ge
    \limsup_{j\to\infty}\Ia(e^{-\varphi^\ast_{i_j}}).
\end{align}
\end{lemma}

\begin{proof}
    The first inequality of \eqref{two-limits} is well-known and appeared in literature many times (see e.g.,  \cite[Lemma 17]{CK2015},  \cite[Lemma 5.4]{FYZZ}, \cite[Lemma 5.8]{HLXZ}). Thus,  we only need to prove the second inequality of \eqref{two-limits}. From \eqref{finite-varphi(x0)} and Lemma
\ref{Le.roc}, one can find a subsequence $\{\varphi_{i_j}\}_{j\in\mathbb{N}}$ of $\{\varphi_i\}_{i\in\mathbb{N}}$ converging to a finite even  convex function $\varphi:M_\mu\to\R$ pointwisely on $M_\mu$ and uniformly on any closed bounded subset of $M_\mu$.   The function $\varphi$ can be easily extended to $\varphi: \mathbb{R}^n\rightarrow \R\cup \{+\infty\}$ by:  
$$\varphi(x):=\lim_{\lambda\to 1^-}\varphi(\lambda x) \ \ \mathrm{if}\   x\in\partial \overline{M}_\mu \ \ \mathrm{and} \ \ \varphi(x):=+\infty \ \mathrm{if} \   
x\notin \overline{M}_\mu. $$ 
Clearly, $\varphi\in  \mathcal{L}_e^+(\mu) \cap \C(\R^n)$. For any $y\in\mathbb{R}^n$, the continuity of $\varphi$ on $\overline{M}_\mu$ yields that 
$$
\varphi^\ast(y)=\sup_{x\in\overline{M}_\mu}\big\{\langle x,y\rangle-\varphi(x)\big\}=\sup_{i\in\mathbb{N}}\big\{\langle x_i,y\rangle-\varphi(x_i)\big\}, 
$$ where $\{x_i\}_{i\in\mathbb{N}}$ is a given dense sequence in $\overline{M}_\mu$. It can be easily checked that,
if we let  $$
\phi^\ast_k(y)=
\max_{i\in[1,k]\cap\mathbb{N}}\big\{\langle x_i,y\rangle-\varphi(x_i)\big\} \ \ \mathrm{for\  each\ } k\in\mathbb{N}, 
$$  then $e^{-\phi^\ast_k}\in\LC$ for $k\in \mathbb N$ big enough (hence $\Ia(e^{-\phi^\ast_k})<\infty$ by Lemma \ref{finiteness of I_q}). Moreover,  $\{\phi^\ast_k\}_{k\in \mathbb N}$ is a monotone increasing  sequence  and converges pointwisely  to $\varphi^*$. By the monotone convergence theorem, 
$$
\lim_{k\to\infty}\Ia(e^{-\phi_k^\ast})
=\Ia(e^{-\varphi^\ast}). 
$$ For $\varepsilon>0$, there exists $k_0\in\mathbb{N}$ big enough such that,
\begin{align}\label{epsilon1}
    \Ia(e^{-\varphi^\ast})\ge \Ia(e^{-\phi^\ast_{k_0}})-\varepsilon.
\end{align} Note that the pointwise convergence of   $ \varphi_{i_j} \rightarrow \varphi$ in $\{x_1,\cdots,x_{k_0}\}$ implies the existence of $j_0\in \mathbb N$ big enough, such that, $\varphi^\ast_{i_j}\ge\phi^\ast_{k_0}-\varepsilon$ for any $j>j_0$.  Then Proposition \ref{properties-a-energy} implies 
\begin{align}\label{epsilon2}
\Ia(e^{-\phi^\ast_{k_0}})
\ge \Ia(e^{-\varphi^\ast_{i_j}})e^{-2\varepsilon}.
\end{align}
Combining \eqref{epsilon1} and \eqref{epsilon2}, we have, for any $j>j_0$, 
$$
  \Ia(e^{-\varphi^\ast})\ge
  \Ia(e^{-\varphi^\ast_{i_j}})e^{-2\varepsilon}-\varepsilon.
$$
This gives the second inequality of \eqref{two-limits} as $\varepsilon\rightarrow 0^+$.
\end{proof}

The following lemma shows that  the minimizers of  the optimization problem \eqref{optimiza-convex} are convex and strictly positive  functions. This approach is motived by the recent work \cite{FYZZ}.

\begin{lemma}\label{Le.vaphi(o)=0}
Let $\alpha>0$, $\mu\in\mathscr{M}$, $\varphi\in  \mathcal{L}_e^+(\mu) \cap \C(\R^n)$  with  $\varphi(o)=0$ and $\Ia(e^{-\varphi^\ast})\ge \tau$.   Then
    \begin{align*}      \int_{\mathbb{R}^n}\varphi d\mu
        \ge \frac{c_\mu}{2}\cdot\left(\frac{\Ia(e^{-\varphi^\ast})}{e^2\Ia(e^{-|\cdot|})}\right)^{\frac{1}{n+\alpha}}- |\mu|,
    \end{align*}
where $c_\mu>0$ is given in \eqref{not-concentrated} and $|\mu|=\int_{\Rn}\,d\mu.$
\end{lemma}

\begin{proof}
For any given $\varphi\in  \mathcal{L}_e^+(\mu) \cap \C(\R^n)$, let  $$r_\varphi=\min_{u\in \sphere} \sup_{x\in K_\varphi} \langle x, u\rangle =\sup_{x\in K_\varphi} \langle x, u_0\rangle $$ with $u_0\in S^{n-1}$ fixed, and  $
K_\varphi=\{x\in\mathbb{R}^n: \varphi(x)\le 1\}
$   the origin-symmetric convex set in $\R^n$. For $\mu\in \mathscr{M}$, it is clear that  the origin  $o\in M_{\mu}$, and hence  $\varphi$ is finite near  $o$ due to $\varphi\in  \mathcal{L}_e^+(\mu)$. Thus, $o\in {\rm int}(K_\varphi)$ and $r_\varphi>0$. It follows from \eqref{def-dual} that, for any $y\in\mathbb{R}^n$, 
$$
\varphi^\ast(y)=\sup_{x\in\mathbb{R}^n}\{\langle x,y\rangle -\varphi(x)\}
\ge \sup_{x\in r_\varphi B_2^n}\{\langle x,y\rangle -\varphi(x)\}
\ge r_\varphi|y|-1. 
$$ By Proposition \ref{properties-a-energy},  one has 
$$
\Ia(e^{-\varphi^\ast})\le e^2 \Ia(e^{-r_\varphi|\cdot|})=r_{\varphi}^{-(n+\alpha)}e^2\Ia(e^{-|\cdot|}).
$$ After a rearrangement of the above formula, one can get \begin{align}\label{r1}
\frac{e^2\Ia(e^{-|\cdot|})}{\Ia(e^{-\varphi^\ast})}\ge r_\varphi^{n+\alpha}.
\end{align} 
A standard argument by the convexity of $\varphi$ and $\varphi(o)=0$ shows that,  if  $|\langle x,u_0\rangle|>2 r_\varphi$, then $\frac{2r_\varphi}{|\langle x,u_0\rangle|} x \notin  K_\varphi$ and 
\begin{align*}
    \frac{2r_\varphi}{|\langle x,u_0\rangle|}\varphi(x)
    \ge \varphi\left(\frac{2r_\varphi}{|\langle x,u_0\rangle|}\cdot x+\left(1-\frac{2r_\varphi}{|\langle x,u_0\rangle|}\right)\cdot o\right)
    = \varphi\left(\frac{2r_\varphi}{|\langle x,u_0\rangle|} x\right)>1.
\end{align*} As $\varphi$ is non-negative, then   $
 (\varphi(x)+1)2r_\varphi \ge|\langle x,u_0\rangle|
$ for any $x\in \Rn$. Together with \eqref{not-concentrated}, one has
\begin{align}\label{r2}
    r_\varphi\ge
    \frac{\int_{\mathbb{R}^n}|\langle x,u_0\rangle|dx}{2\int_{\mathbb{R}^n}\varphi d\mu+2|\mu|}
    \ge  \frac{c_\mu}{2\int_{\mathbb{R}^n}\varphi d\mu+2|\mu|}.
\end{align} The desired inequality follows immediately from inequalities \eqref{r1} and \eqref{r2}.
\end{proof}
 
Now, we can finish the proof of  Proposition \ref{existence-of-function}. 

\begin{proof}[Proof of Proposition \ref{existence-of-function}] For any $t>0$,  Proposition \ref{properties-a-energy} implies that \begin{align}\label{homogenenity}
  \Ia(\mathbf{1}_{t B_2^n})= \Ia(\mathbf{1}_{B^n_2}(t^{-1}x)) =t^{n+\alpha}\Ia(\mathbf{1}_{B_2^n}).
\end{align}  Due to \eqref{homogenenity} and the finiteness of $\Ia(\mathbf{1}_{B_2^n})$, the map  $t\mapsto \Ia(\mathbf{1}_{tB_2^n})$ is clearly continuous, and 
$$\lim_{t\to 0^+}\Ia(\mathbf{1}_{ t B_2^n})=0  \ \ \mathrm{and} \ \ \lim_{t\to +\infty }\Ia(\mathbf{1}_{ t B_2^n})=+\infty.$$ Let $t_0\in (0, +\infty)$ be such that $\Ia(\mathbf{1}_{{t_0}B_2^n})=1$. Then, from \eqref{first-moment}, \begin{align}
\int_{\mathbb{R}^n}\breve{\phi}(x)d\mu=\frac{|\mu|}{2}\log \tau+t_0\int_{\mathbb{R}^n}|x|d\mu<\infty
\ \ \text{and}
\ \ \Ia(e^{- \breve{\phi}^\ast})=\tau,  \label{relation-phi}
\end{align} where  $ \breve{\phi}(x)=\frac{1}{2}\log \tau +t_0|x|$ for all $x\in\R^n$.  Therefore, \eqref{optimiza-convex} is well-defined. In particular, $\Theta_{\tau}\in [0, \infty)$, and a limit sequence $\{\varphi_i\}_{i\in\mathbb{N}}$ with $\varphi_i\in  \mathcal{L}_e^+(\mu) \cap\C(\R^n)$ can be found so that  $\Ia(e^{-\varphi_i})\ge \tau$  and  $\int_{\mathbb{R}^n}\varphi_id\mu\to\Theta_{\tau}$ as $i\to\infty$. The latter further implies
$$
\sup_{i\in\mathbb{N}}\int_{\mathbb{R}^n}\varphi_id\mu<\infty.
$$ By Lemma \ref{Le.existence}, there exists  $\varphi \in  \mathcal{L}_e^+(\mu) \cap \C(\R^n)$ satisfying \eqref{two-limits}. In particular, $\Ia(e^{-\varphi^\ast})\ge \tau$.  

  In fact, the function $\varphi$ found above is strictly positive when the given constant $\tau$ is big enough. To this end, if $\varphi(o)=0$,  then   Lemma \ref{Le.vaphi(o)=0} shows that     
$$
\int_{\mathbb{R}^n}\varphi d\mu
\ge c_{n,\alpha,\mu}\cdot \tau^{\frac{1}{n+\alpha}}
-|\mu|,
$$
where  $c_{n,\alpha,\mu}=\frac{c_\mu}{2}\cdot (e^2\Ia(e^{-|\cdot |}))^{-\frac{1}{n+\alpha}}>0$ is a constant  independent of $\tau$. Due to \eqref{relation-phi} and  $\breve{\phi}(x)=\frac{1}{2}\log \tau +t_0|x|\in  \mathcal{L}_e^+(\mu)\cap \C(\Rn)$,  one gets  $
\int_{\mathbb{R}^n} \breve{\phi} d\mu\ge\int_{\mathbb{R}^n}\varphi d\mu
$ and hence $$
\frac{(1+\frac{1}{2}\log \tau) |\mu|+M_1}{\tau^{\frac{1}{n+\alpha}}}
\ge c_{n,\alpha,\mu},
$$
where $M_1=t_0\int_{\mathbb{R}^n}|x|d\mu<\infty$. Clearly, 
$$\lim_{\tau \to+\infty}\frac{(1+\frac{1}{2}\log \tau) |\mu|+M_1}{\tau^{\frac{1}{n+\alpha}}}=0,$$
which is impossible as $c_{n,\alpha,\mu}>0$. This contradiction shows that $\varphi(o)>0$ and hence  $\varphi$  is strictly positive as $\varphi \in  \mathcal{L}^+_e(\mu)\cap\C(\Rn)$. Consequently, for a  sufficiently small
 $\varepsilon>0$,  $\varphi_\varepsilon(x)=\varphi(x)-\varepsilon$ is an even non-negative convex function and
$\Ia(e^{-\varphi_\varepsilon^\ast})=
 \Ia(e^{-\varphi^\ast})e^{-2\varepsilon}$.    Thus, if  $\Ia(e^{-\varphi^\ast})>\tau$, then   
 $\Ia(e^{-\varphi_\varepsilon^\ast})\geq \tau$ for $\varepsilon>0$ small enough and  $$
\int_{\mathbb{R}^n}\varphi_\varepsilon(x)d\mu<\int_{\mathbb{R}^n}\varphi(x)d\mu.
$$ This contradicts to the minimality of $\varphi$, and hence $\Ia(e^{-\varphi^\ast})=\tau.$

In summary,  $\varphi\in  \mathcal{L}_e^+(\mu)\cap \C(\Rn)$ is strictly positive, $\Ia(e^{-\varphi^\ast})\!=\!\tau$ and by \eqref{two-limits}, $$ \Theta_{\tau}\leq  \int _{\Rn} \varphi \,du\leq \lim_{i\to \infty}\int _{\Rn} \varphi_i\,du= \Theta_{\tau}.$$ That says, the function $\varphi$ solves the optimization problem \eqref{optimiza-convex} (and \eqref{optimiza-pro}). 
\end{proof}

We will need the following lemma. 
\begin{lemma}
Let $f=e^{-\varphi}\in \LC$ and $\alpha>0$.  If  $\zeta:\mathbb{R}^n\to\mathbb{R}$ is an even, compactly supported, continuous function, then
\begin{align}\label{derivat-Iqvarphit}
\frac{d\Ia(e^{-(\varphi+t\zeta)^\ast})}{dt}\bigg|_{t=0}
=2\int_{\mathbb{R}^n}\zeta(\nabla\varphi^\ast(x)) I_{\alpha}(e^{-\varphi^*},x)e^{-\varphi^*(x)}dx.
\end{align}
\end{lemma}
\begin{proof} The condition on $\zeta$ guarantees the existence of $M>0$ such that $|\zeta(x)|\le M$ for all $x\in\mathbb{R}^n$. This, together with \eqref{def-dual}  ensures  that
$$
 \varphi^*-|t|M\leq (\varphi+t\zeta)^*\leq \varphi^*+|t|M.
$$  This further gives
\begin{align*}
    \bigg|\frac{e^{-(\varphi+tg)^*}-e^{-\varphi^\ast}}{t}\bigg|
    \le e^{-\varphi^\ast}\cdot\max\left\{ \left|\frac{e^{tM}-1}{t}\right|,\left|\frac{e^{-tM}-1}{t}\right|\right\}.
\end{align*}
 The maximum on the right-hand side is uniformly bounded  when $|t|>0$ is small enough. It follows from Lemma \ref{rotem-111}, Proposition \ref{finiteness of I_q}, Fubini's theorem, and the dominated convergence theorem that 
 \begin{align*} 
\lim_{t\to0}\frac{\Ia(e^{-(\varphi+t\zeta)^\ast})-\Ia(e^{-\varphi^\ast})}{t}
&= \int_{\R^n}\int_{\R^n}\lim_{t\to0}\frac{e^{-(\varphi+t\zeta)^\ast(x)}e^{-(\varphi+t\zeta)^\ast(y)}-e^{-\varphi^\ast(x)}e^{-\varphi^\ast(y)} }{t|x-y|^{n-\alpha}}dxdy\\
&=2\int_{\R^n}\int_{\R^n}\frac{\zeta(\nabla\varphi^*(x))  e^{-\varphi^*(x)}e^{-\varphi^*(y)}}{|x-y|^{n-\alpha}}dxdy\\&=2\int_{\R^n} \left(\int_{\R^n}\frac{e^{-\varphi^*(y)}}{|x-y|^{n-\alpha}}dy\right) \zeta(\nabla\varphi^*(x))  e^{-\varphi^*(x)}dx\\ &=2\int_{\mathbb{R}^n}\zeta(\nabla\varphi^\ast(x)) I_{\alpha}(e^{-\varphi^*},x)e^{-\varphi^*(x)}dx,
\end{align*} where the third equality fillows from Fubini theorem and the last equality follows from \eqref{Riesz-P-def-1}. This completes the proof.  
\end{proof}

We now prove our main result Theorem \ref{Solution-a-R-e}, which is stated here for  convenience. 

\begin{theorem}
Let $\alpha>0$ and $\mu\in\mathscr{M}$.  There exists a log-concave function $f\in \LC$ with $\Ia(f)=|\mu|$, such that, $
   \mu=\Rma(f,\cdot). $
\end{theorem}
\begin{proof} Proposition \ref{existence-of-function} guarantees the existence of a constant $\tau_0>0$, and  an even, lower semi-continuous, convex function $\varphi_0\in \mathcal{L}_e^+(\mu)\cap \C(\Rn)$  with  $\varphi_0>0$  and $\Ia(e^{-\varphi_0^\ast})=\tau_0$ such that
\begin{align*}  
\int_{\mathbb{R}^n}\varphi_0 d\mu =\inf\left\{\int_{\mathbb{R}^n}\varphi d\mu:
    \varphi\in \mathcal{L}_e^+(\mu)\ \ \text{and}\ \  \Ia(e^{-\varphi^\ast})\ge \tau_0 \right\}.  \end{align*}  
 Let $\zeta:\mathbb{R}^n\to\mathbb{R}$ be an even, compactly supported, continuous function. Then $|\zeta(x)|\le M$ for all $x\in\mathbb{R}^n$ with $M>0$ a constant. Moreover, a suitable  constant $t_0\in(0,1)$ can be found so that  $\varphi_t=\varphi_0+t\zeta\in \mathcal{L}_e^+(\mu)$ and $\varphi_0-2|t|M\geq 0$ for any $t\in[-t_0,t_0]$. Define $\phi_t$ by
\begin{align*}
    \phi_t=\varphi_t-\frac{1}{2}\log \Ia(e^{-\varphi_t^\ast})+\frac{1}{2}\log \tau_0.
\end{align*} It is easily checked that, for any $t\in[-t_0,t_0]$ and $x, y\in \R^n$, \begin{align*}
 |\varphi_t(x)-\varphi_0(x)| \leq |t|M \ \ \mathrm{and} \ \
|\varphi_t^\ast(y)-\varphi_0^\ast(y)|\le|t|M. 
\end{align*} In particular,  $\Ia(e^{-\phi_t^\ast})=\tau_0$ and by Proposition \ref{properties-a-energy}, 
$$
\phi_t\geq \varphi_0-|t|M-\frac{1}{2}\log \Ia(e^{-\varphi_0^\ast+|t|M})+\frac{1}{2}\log \tau_0 = \varphi_0-2|t|M\geq 0.
$$
 That says,  $\phi_t$ satisfies the conditions in the optimization problem (\ref{optimiza-pro}) with $\tau_0$ and $\phi_0=\varphi_0$. Thus,   
\begin{align*}
    \frac{d}{dt}\int_{\mathbb{R}^n}\phi_td\mu\bigg|_{t=0}=0.
\end{align*} Employing   \eqref{def-Rma-22} and (\ref{derivat-Iqvarphit}), one can easily get 
\begin{align*}
   0&= \int_{\mathbb{R}^n}\zeta(x)d\mu
    -\frac{|\mu|}{\Ia(e^{-\varphi_0^\ast})}\int_{\mathbb{R}^n}
    \zeta(\nabla\varphi_0^* (x)) I_{\alpha}(e^{-\varphi_0^*},x)e^{-\varphi_0^*(x)}dx\nonumber\\
&=\int_{\mathbb{R}^n}\zeta(x)d\mu
    - \frac{|\mu|}{\Ia(e^{-\varphi_0^\ast})} \int_{\mathbb{R}^n}
    \zeta(z) \,d\Rma(e^{-\varphi_0^*}, z)dz
\end{align*} holds for any even, compactly supported, continuous function $\zeta$. Thus, $$\,d\mu=\frac{|\mu|}{\Ia(e^{-\varphi_0^\ast})}\,d\Rma(e^{-\varphi_0^*}, z).$$ Let $f= \left(\frac{  |\mu| }{\Ia( e^{-\varphi_0^{\ast}})}\right)^{1/2}e^{-\varphi_0^*}$. It follows from \eqref{Rma-homogeneous} that $\mu= \Rma (f,\cdot)$ and $\Ia(f)=|\mu|$ by Proposition \ref{properties-a-energy}.  That is, $f$ is a solution to the  Riesz $\alpha$-energy Minkowski problem for log-concave functions.  
\end{proof}

\vskip 2mm \noindent  {\bf Acknowledgement.}  The
research of NF was supported by  NSFC (No.12001291) and the Fundamental Research Funds for the Central Universities (No.531118010593). The
research of DY was supported by a NSERC grant, Canada.
 The research of ZZ was supported by NSFC (No. 12301071), Natural Science Foundation of
Chongqing, China CSTC (No. CSTB2024NSCQ-MSX1085) and the Science and Technology Research Program of Chongqing Municipal
Education Commission (No. KJQN202201339).

\noindent Niufa Fang, School of  Mathematics, Hunan University, Changsha, Hunan, 410082, China\\
\textit{Email address}: fangniufa@hnu.edu.cn
\vspace{8pt}

\noindent Deping Ye, Department of Mathematics and Statistics, Memorial University of Newfoundland, St. John’s, Newfoundland, A1C 5S7, Canada\\
\textit{Email address}: deping.ye@mun.ca
\vspace{8pt}

\noindent Zengle Zhang, Chongqing Key Laboratory of Group and Graph Theories and Applications, Chongqing University of Arts and Sciences, Chongqing, 402160, China\\
\textit{Email address}: 
zhangzengle128@163.com
\vspace{8pt}

\end{document}